\documentclass[a4paper, 11pt,reqno]{amsart}
\usepackage[english]{babel}
\usepackage[latin1]{inputenc}
\usepackage{eucal}
\usepackage{mathrsfs,amsmath,amssymb,latexsym,amsthm,amsfonts,mathtools}
\usepackage[all,cmtip]{xy}
\usepackage[usenames,dvipsnames,svgnames,table]{xcolor}
\usepackage{enumitem}
\usepackage[plainpages=false,colorlinks,pdfpagelabels]{hyperref}

\usepackage[hmargin=3.0cm,vmargin=3.5cm]{geometry}

\title{
  Relations between different types of Hypoellipticity: A systematic approach
}

\author{Bruno de Lessa Victor}
\address{Universidade Federal de Santa Catarina, Brazil}
\email{\texttt{bruno.lessa@ufsc.br}}

\author{Luis F.~Ragognette}
\address{Universidade Federal de Minas Gerais, Brazil}
\email{\texttt{luisragognette@mat.ufmg.br}}

\thanks{Supported   by the
  Conselho Nacional de Desenvolvimento Cient{\'i}fico e Tecnol{\'o}gico
  (CNPq, grants~404175/2023-6).
}


\date{}
 
\theoremstyle{definition} 

\newtheorem{Def}{Definition}[section]

\theoremstyle{definition}
\newtheorem{Exe}[Def]{Example}
\newtheorem{Obs}[Def]{Remark}

\theoremstyle{plain}
\newtheorem{Pro}[Def]{Proposition}
\newtheorem{Cor}[Def]{Corollary}
\newtheorem{Teo}[Def]{Theorem}
\newtheorem{Lem}[Def]{Lemma}

\newcommand{\ord}{\mathrm{ord}}

\newcommand{\Span}{\mathrm{span}}

\newcommand{\transp}[1]{\prescript{\mathrm{t} \!}{}{{#1}}}

\newcommand{\XX}{\mathrm{X}}
\newcommand{\LL}{\mathrm{L}}
\newcommand{\MM}{\mathrm{M}}

\newcommand{\loc}{\mathrm{loc}}

\newcommand{\RN}{\R^N}

\newcommand{\Rm}{\R^m}

\newcommand{\Char}{\mathrm{Char}}

\newcommand{\supp}{\mathrm{supp}\,}

\newcommand{\R}{\mathbb{R}}
\newcommand{\N}{\mathbb{N}}

\newcommand{\Rn}{\R^{n}}
\newcommand{\C}{\mathbb{C}}

\newcommand{\D}{\mathscr{D}}
\newcommand{\E}{\mathscr{E}}

\newcommand{\Z}{\mathbb{Z}}

\newcommand{\Cinf}{\mathcal{C}^{\infty}}
\newcommand{\Com}{\mathcal{C}^{\omega}}

\newcommand{\lra}{\longrightarrow}

\newcommand{\del}{\partial}

\renewcommand{\Re}{\mathrm{Re }}
\renewcommand{\Im}{\mathrm{Im }}
\newcommand{\dd}{\textnormal{d}}
\newcommand{\T}{\mathrm{T}'}

\newcommand{\hypo}{\mathfrak{h}}
\newcommand{\dfn}{\doteq}
\newcommand{\LLambda}{\underline{\Lambda}}
\allowdisplaybreaks

\begin{document}

\maketitle
\pagestyle{plain}

\begin{abstract}

We give a systematic treatment to the concept of hypoellipticity, putting it into an abstract form which allows us to deal with several different notions within the same framework. We then investigate when a notion of hypoellipticity implies another one and, in particular, when it can be extended for more general spaces. We also present a relation between certain types of hypoellipticity and local solvability (for the transpose) for a family of operators.

\end{abstract}

\section{Introduction}

The notion of hypoellipticity is central in the study of regularity of differential operators. Highly impactful works related to the concept, such as \cite{batr2},  \cite{hor2}, \cite{hn},  \cite{k3}, \cite{mal} and \cite{rs}, display its importance for several different areas: Partial Differential Equations, Harmonic Analysis, Complex Analysis, Stochastic Analysis, Statistical Mechanics, among others. 
In fact many research papers regarding the hypoellipticity of some models of differential operators have been moving the theory forward.

Usually, one wants to find conditions which guarantee a certain type of hypoellipticity on the operator. Our approach here is slightly different; we systematically treat several different notions of hypoellipticity and look for  answers for the following questions:
\begin{itemize}
	\item When does a type of hypoellipticity imply another type?
	\item When is it possible to prove an equivalence between different types of hypoellipticity? When does this sort of equivalence fail?
	\item How deep are the relations between  hypoellipticity and local solvability?
\end{itemize}

These questions are inspired by famous results in the literature. For instance, for constant coefficient operators it is well-known that hypoellipticity (in the smooth sense) implies the existence of an $s_0\geq 1$ such that the operator is also Gevrey-hypoelliptic of order $s$ for every $s\geq s_0$. Furthermore, if the operator is Gevrey-hypoelliptic of order $s$, for some $s >1$, then it is hypoelliptic in the smooth sense. 

Another classical example derives from  operators of given by ``\emph{real sum of squares}'' satisfying \emph{Hörmander's condition} \cite{hor2}. Thanks to Hörmander's remarkable theorem we know that any operator of such type is hypoelliptic. Beyond that, if its coefficients are real-analytic, a result by Derridj and Zuily~\cite{dz} proves the existence of $s_0 \geq 1$ such that if $s\geq s_0$, it is also Gevrey-hypoelliptic of order $s$. On the other side, Baouendi and Goulaouic presented in \cite{bg} an example of an operator with real-analytic coefficients and satisfying Hörmander's condition which is not analytic-hypoelliptic.  In fact, an important unsolved problem in the theory is a complete characterization of analytic-hypoellipticity for this class of second-order differential operators.

It is worth mentioning that for all the results above, hypoellipticity is studied with the starting point of \emph{distributions}. Our work was strongly motivated by results of Cordaro and Hanges in \cite{ch1};  there the authors study the relation between analytic-hypoellipticity with respect to distributions with the one in the setting of hyperfunctions. One of their results which attracted us the most is the following: let $P$ be a partial differential operator with constant coefficients; then $P$ is \emph{elliptic} if and only if it satisfies the following property:
\begin{align*}
&\text{\emph{the only hyperfunctions that $P$ sends in the space}} \\
&\text{\emph{of smooth functions are already smooth functions},}
\end{align*}
which is quite different in comparison with the case of distributions, for instance, the \emph{heat operator} is a famous example of a hypoelliptic operator that is not elliptic. Furthermore, this result shows that for a generic differential operator with real-analytic coefficients one cannot expect to extend hypoellipticity from distributions up to hyperfunctions.

An unusual and interesting case is considered in \cite{gz}, where the authors study analytic-hypoellipticity by understanding how the Keldysh operator acts on \emph{smooth functions}, instead of distributions or hyperfunctions. In this context, it seems natural to ask when these three notions are equivalent. Generally speaking, this type of question can be made in a much broader sense, taking into consideration  \emph{different classes of differential operators acting on several spaces of functions and generalized functions}. 

This is the main theme of this work, which is organized as it follows: in Section~\ref{section2} we provide basic definitions to deal with hypoellipticity in an abstract manner. In Section~\ref{constant coefficients operators} we prove, in the context of differential operators with constant coefficients, the equivalence of different notions of hypoellipticity in the Gevrey framework. In particular, we provide in Theorem~\ref{teo1} a Gevrey version for the aforementioned result proved in \cite{ch1}. Among other things, our result says the following: if a constant coefficient operator $P$ fails to be Gevrey-hypoelliptic of order $s$,  
\begin{itemize}
	\item for every $t>s$ there exists $g$ a Gevrey function of order $t$ but not of order $s$ such that $Pg$ is a Gevrey function of order $s$.
\end{itemize}

In Section~\ref{section4} we consider operators that commute with an elliptic operator; under this condition we investigate hypoellipticity regarding Sobolev spaces. As a consequence Corollary~\ref{operators not hypoelliptic}, one has the following: if $P$ is such an operator which fails to be hypoelliptic,  one has the following properties:
\begin{itemize}
	\item for every $k\in \N$ there exists a non-smooth function $g$ of class $\mathcal{C}^{k}$ such that $Pg$ is a smooth function;
	\item for every $r\in \R$ there exists a distribution $u$ which does not belong to $H_{\textrm{loc}}^{r}$ such that $Pu$ is a smooth function.
\end{itemize}
It is worth mentioning that an important example is the celebrated Lewy's vector field, presented first in \cite{l}. 

In Section~\ref{section5} we explore the idea of extending hypoellipticity; first we recall some results by Komatsu~\cite{kom2} in one dimension, where we have a class of operators which are analytic-hypoelliptic when the starting point is a certain space of Gevrey functions of order $s$, but it fails to be analytic hypoelliptic if the starting point is replaced by another space of Gevrey functions of a certain order $t$, with $t > s$. Next we  prove  a result in an opposite direction, for a certain model of operators  which are named \emph{of tube type}. If $P$ is a differential operator of such type, we prove the following:
\begin{itemize}
	\item Let $1<s<r$; in order to check whether $P$ is Gevrey-hypoelliptic of order $s$ with respect to ultradistributions of order $s$, it is sufficient to verify that $P$ does not sends any Gevrey function of order $r$ into a function of order $s$, unless it was previously of order $s$.
\end{itemize}  
Even though the result is stated and proved in the Gevrey context, one can adapt the result to analytic case as well, as pointed out in Remark~\ref{RA case}. An important example in our model is the previously mentioned Baouendi-Goulaouic operator in $\R^{3}$, given by
\begin{equation*}
	Q= \frac{\del^{2}}{\del t^{2}} + t^{2} \frac{\del^{2}}{\del x_{1}^{2}} + \frac{\del^{2}}{\del x_{2}^{2}}.
\end{equation*}
It was proved for instance by Christ in \cite{c} and by Bove and Tartakoff in \cite{bt} that $Q$ is Gevrey-hypoelliptic of order $s$ with respect to distributions if and only if $s \geq 2$. Hence  Theorem~\ref{Gevrey hypo implies ultra hypo} implies that
\begin{itemize}
\item for every $s>1$, it is always possible to find an open set $U \subset \R^{3}$ and $f \in G^s(U) \setminus \Com(U)$ such that $Qf \in \Com(U)$;
\item for every $s \geq 2$, let  $u$ be an ultradistribution  of order $s$  for which $Pu$ is a Gevrey function of order $s$. Then $u$ is actually a Gevrey function of order $s$.  
\end{itemize}

In Section~\ref{section6}, we study in a model of operator that was previously studied in \cite{ch1,ch2}. We prove in Theorem~\ref{weak hypo} that for any such operator which is Gevrey-hypoelliptic of order $s$ with respect to distributions, one has different notions of hypoellipticity for its \emph{transpose}. In particular, our result can be applied for operators given by sum of squares of real vector fields that satisfy Hörmander's condition and have real-analytic coefficients.  In the particular case of the Baouendi-Goulaouic operator, a combination of Theorem~\ref{weak hypo} and \cite[Corollary $1.1$]{ch2} implies that
\begin{itemize}
	\item $Q$ is hypoelliptic with the starting point given by ultradistributions of order $s$ \emph{if and only if $s \geq 2$}.
\end{itemize}

  Hypoellipticity for system of vector fields is considered in Section~\ref{LIS Gevrey}. More precisely, we prove an equivalence between different notions of Gevrey hypoellipticity for systems of vector fields associated to a locally integrable structure. Among the operators for which this result holds true we have the Mizohata structures in general. 
  
It is well established that the standard hypoellipticity of an operator implies local solvability of its transpose. In Section~\ref{section8}, we prove in Theorem~\ref{solution delta type} that different notions of hypoellipticity imply local solvability, in the smooth sense, for the transpose operator.
 
Finally, in the Appendix we discuss some results on infinite order differential operators that are important in the Gevrey/analytic setting. In particular they allow us to represent ultradistributions/hyperfunctions as an infinite order operator applied to a smooth function. Furthermore, a great part of them have a notion of strong ellipticity, which guarantees that any element in their kernel are in fact real-analytic functions. Not only these results are useful by themselves, but they are also indispensable as technical tools for proofs in Sections~\ref{constant coefficients operators}, \ref{section5} and \ref{LIS Gevrey}.

\section{Hypoellipticity in an abstract setup}\label{section2}


\subsection{Hypoellipticity of pairs}

Let us start defining  hypoellipticity in an abstract manner. Let $\mathcal{E}$ be a sheaf of topological vector spaces over $\Omega$ and consider $P$ a morphism of $\mathcal{E}$, meaning that
\begin{equation*}
\text{$P: \mathcal{E}(U) \lra \mathcal{E}(U)$ is well defined for every $U$ open subset of $\Omega$,}
\end{equation*}
and if $\iota_{UV}: \mathcal{E}(U) \lra \mathcal{E}(V)$ is the restriction from an open set $U$ to an open subset $V$, then $P\circ \iota_{UV}= \iota_{UV} \circ P$.
Let $\mathcal{F}$ and $\mathcal{G}$ be two subsheaves of $\mathcal{E}$ and assume that $\mathcal{G}$ is a subsheaf of $\mathcal{F}$.

\begin{Def}
 The {\it hypoellipticity of $P$ with respect to the pair $(\mathcal{F},\mathcal{G})$} is described as it follows: for every $U$ open subset of $\Omega$, $P$ satisfies the following property:
\begin{align}\label{hipodepravado}
  \textrm{ whenever  $u\in \mathcal{F}(U)$ satisfies $Pu\in \mathcal{G}(U)$, then } u\in \mathcal{G}(U).
\end{align}
When $P$ is hypoellipticity  with respect to the pair $(\mathcal{F},\mathcal{G})$ we denote that $P$ is $\mathfrak{h}(\mathcal{F},\mathcal{G})$
\end{Def}

Since we are interested in partial differential equations, we will always assume that the sheaves are subsheaves of the \emph{sheaf of hyperfunctions} (our approach to hyperfunctions follow \cite{tre}, see also \cite{ct2}, \cite{sch}) and the morphisms in question are given by the action of partial differential operators, whose coefficients are sufficiently regular so each sheaf in the pair is preserved.
More precisely, we consider the following sheaves on $\R^{n}$:
\begin{itemize}
	\item The \emph{sheaf of real-analytic functions}, which will be denoted by $\Com$, and the \emph{sheaf of hyperfunctions}, here denoted by  $\mathcal{B}$; 
\item The  \emph{sheaves of smooth functions, localized Sobolev spaces of order $k$ and  distributions}, which will be denoted respectively by $\Cinf$, $H_{\textrm{loc}}^{k}$ and  $\D'$;
	\item The \emph{sheaves of Gevrey functions and ultradistributions of Roumieu type order $s$} for any $s > 1$, which will be denoted respectively by $G^{s}$ and $\D_{s}'$;
	\item The \emph{sheaves of Gevrey functions and ultradistributions of Beurling type of order $s$} for any $s > 1$, which will be denoted respectively by $G^{(s)}$ and $\D_{(s)}'$.
\end{itemize} 

In this context, to say that $P$ is hypoelliptic in the classical sense is equivalent to state that  $P$ is  $\hypo(\D', \Cinf)$. Similarly, to say  that $P$ is  $\hypo(\D', \Com)$  is the same as to state that $P$ is analytic-hypoelliptic in the classical sense. 

A simple observation is that if hypoellipticity fails, the elliptic regularity theorem implies it fails on the characteristic set of the differential operator. Moreover, let $P$ be a differential operator with real-analytic coefficients on  an open subset $U$  of $\Rn$; if $\mathcal{S}$ is any subsheaf of $\D'$ which contains $\Com$ as a subsheaf,  it was proved in \cite{hr} that
\begin{align*}
  WF_{\mathcal{S}}(u) \subset WF_{\mathcal{S}}(Pu) \cup \Char P, \quad \forall u \in \D'(U)
\end{align*}
if and only if
\begin{align*}
  WF_{\mathcal{S}}(u) \subset WF_{\mathcal{S}}(Pu) \cup \Char P, \quad \forall u \in \mathcal{B}(U).
\end{align*}
Hence if $P$ is $\hypo(\mathcal{B}, \mathcal{S})$, it is also $\hypo(\D', \mathcal{S})$ and the obstruction for the converse happens on $\Char~P$.
In particular, this result implies that if $P$ is elliptic (i.e., $\Char~P=\emptyset$), the properties $\hypo(\mathcal{B}, \mathcal{S})$ and $\hypo(\D', \mathcal{S})$ are equivalent.

In general one may expect the following: if $P$ is an operator whose coefficients are \emph{sufficiently regular} any type of the obstruction to navigate between different notions of hypoellipticity will happen on the characteristic set.

\begin{Obs}\label{RemarkHypoLocal}
One thing to keep in mind is that the problem is \emph{local} in the following sense: given $u \in \mathcal{F}(U)$ satisfying $Pu \in \mathcal{G}(U)$ is enough to find a cover of $\{V_k\}$ of $U$ and check that $u|_{V_k} \in \mathcal{G}(V_k)$ for every $k$. 
\end{Obs}

Certain arguments are pretty general and for this reason it is convenient to approach them with the abstract setup just stated. For instance,  consider $\mathcal{F}, \mathcal{G}, \mathcal{H}$ sheaves such that $\mathcal{H}\subset \mathcal{G}\subset \mathcal{F}$, where each inclusion means \emph{that one is a subsheaf of the other}. It is not difficult to check that we have the following statement. 

\begin{Pro}
 Let $\mathcal{F}, \mathcal{G}$ and $\mathcal{H}$ be sheaves and $P$ an operator as above. We have the following relations:  
  \begin{enumerate}[leftmargin=*]
  \item If $P$ is $\hypo(\mathcal{F},\mathcal{G})$ and $\hypo(\mathcal{G}, \mathcal{H})$ then $P$ is $\hypo(\mathcal{F},\mathcal{H})$; \label{cond1}
  \item If $P$ is $\hypo(\mathcal{F}, \mathcal{H})$ then $P$ is $\hypo(\mathcal{G}, \mathcal{H})$. \label{cond2}
  \end{enumerate}
\end{Pro}

\subsection {Hypoellipticity of triples}

Sometimes it may be interesting to consider a notion of hypoellipticity which takes into consideration three sheaves. Consider for instance the following result obtained by J.J. Kohn in  \cite{k}: for each integer $k> 1$ the author constructs a differential operator $E_{k}$ which satisfies the following condition:
\begin{equation*}
\forall U \subset \subset \R^{3}, \ \ \forall u \in \D'(U), \ \ E_{k} u \in L_{\textrm{loc}}^{2}(U) \ \text{\emph{implies}} \ u \in H_{\textrm{loc}}^{1 - k}(U).
\end{equation*}

Observe that this property cannot be described as a hypoellipticity of pairs, since  $H_{\textrm{loc}}^{1-k}(U) \subsetneq L_{\textrm{loc}}^{2}(U)$ whenever $k > 1$. This type of example leads us to an abstract concept of \emph{hypoellipticity of triples}.

\begin{Def}
  Let $\mathcal{E}$ be a sheaf of topological vector spaces and $P$ a morphism of $\mathcal{E}$.
Consider three sheaves $\mathcal{H}\subset \mathcal{G}\subset \mathcal{F}$ over an open subset $\Omega$ of $\RN$. We shall say that $P$ is $\hypo(\mathcal{F}, \mathcal{G}, \mathcal{H})$ on $\Omega$ if for every open subset $U\subset \Omega$, the following condition is satisfied:
	\begin{align*}
		u\in \mathcal{F}(U) \ \textrm{and} \ Pu \in \mathcal{H}(U) \ \Longrightarrow \ u \in \mathcal{G}(U).
	\end{align*}
\end{Def}

\begin{Obs}
	Any operator $P$ which is $\hypo(\mathcal{F}, \mathcal{G})$ is trivially $\hypo(\mathcal{F}, \mathcal{G}, \mathcal{H}).$
\end{Obs}

\section{Hypoellipticity for constant coefficient operators in the Gevrey context} \label{constant coefficients operators}

In this section we study  relations regarding  hypoellipticity of pairs in the real-analytic and Gevrey frameworks for operators with constant coefficients. First we recall a result about equivalent notions of analytic-hypoellipticity.
\begin{Teo} [Theorem 1 - \cite{ch1}]  \label{Theorem 1}
  Let $P$ be a constant coefficient differential operator in $\RN$. The following properties are equivalent:
  \begin{enumerate}
  \item $P$ is elliptic;
  \item $P$ is $\hypo(\D', \Com)$; \label{eq2}
  \item $P$ is $\hypo(\mathcal{B}, \Com)$;\label{eq3}
  \item $P$ is $\hypo(\mathcal{B}, \Cinf)$. \label{eq4}
  \end{enumerate}
\end{Teo}

Next  result shows that we can  add two new conditions to their theorem.

\begin{Teo}\label{Teoelliptichyperfunctions}
	Let $P$ be a constant coefficient differential operator in $\RN$. The following  are equivalent:
	\begin{enumerate}
		\item $P$ is elliptic; \label{eq1} \setcounter{enumi}{4}
		\item $P$ is $\hypo(\mathcal{B}, \D')$; \label{eq5}
		\item $P$ is $\hypo(\Cinf, \Com ).$ \label{eq6}
	\end{enumerate}
\end{Teo}
\begin{proof}
  The fact that \eqref{eq5} implies \eqref{eq1} is an immediate consequence of \cite[Corollary 2, pg. 103]{sch}. On the other hand,   \eqref{eq5} follows  directly from \eqref{eq1} thanks to the \emph{elliptic regularity theorem} (see \cite[Theorem 6.1]{hr}),  which shows the equivalence of statements \eqref{eq1}-\eqref{eq5}. Furthermore, it is immediate that   \eqref{eq6} follows from \eqref{eq3}. 
  
  Now assume \eqref{eq6}; let $U \subset \RN$ be an open set and $u\in \mathcal{B}(U)$ such that $Pu \in \Com(U)$. Fix $x_{0}$ in $U$; by Remark \ref{RepresentacaodeHiperfuncoes}, there exist $V\subset\subset U$ an open neighborhood of $x_{0}$,  a \emph{strongly elliptic} hyperdifferential operator $Q$ of type $\{p!\}$ (see Definitions \ref{symbols ultradiff} and \ref{Strongly elliptic}), $v\in \Cinf(V)$ and $g\in \Com(V)$ such that
  \begin{equation*}
	u = Qv \ \text{and} \ Pu = Qg \ \text{on $V$. }
  \end{equation*}
 Since $P$ has constant coefficients, it commutes with $Q$, and therefore
\begin{align*}
	Q(Pv-g)= PQv- Qg= Pu-Pu=0 \ \ \text{on $V$.}
\end{align*}
As a consequence of Remark~\ref{kernel analytic case} one has that $Pv-g\in \Com(V)$, which implies  that $Pv\in \Com(V)$. Since $P$ is $\hypo(\Cinf, \Com)$ it follows that $v\in \Com(V)$; from the continuity of $Q$ on $\Com(V)$ we infer that $u$ is real-analytic on an open neighborhood of $x_{0}$. Due to the fact that $x_{0}$ is arbitrary, we conclude that $P$ is $\hypo(\mathcal{B}, \Com)$.
\end{proof}

Next we consider the same problem in the Gevrey setting.

\begin{Teo}\label{teo1}
  Let $P$ be a constant coefficient differential operator in $\RN$. For any $s> 1$, we have the following equivalences:
  \begin{enumerate} 
  \item $P$ is $\hypo(\D'_s, G^{s})$; \label{eqs1}
  \item $P$ is $\hypo(\D', G^s)$; \label{eqs2}
  \item $P$ is $\hypo( \Cinf, G^s)$; \label{eqs3}
      \item $P$ is $\hypo( G^t, G^s)$ for $t>s$; \label{eqs3'}
  \item $P$ is $\hypo( \D'_s, \D')$; \label{eqs4}
  \item $P$ is $\hypo( \D'_s, \Cinf)$. \label{eqs5}
  \end{enumerate}
\end{Teo}
The proof of Theorem \ref{teo1} has several steps and will require results related to infinite order differential operators.
 It is clear that if we prove that $\eqref{eqs1}$ is consequence of $\eqref{eqs3'}$, then the first four statements are equivalent.

  \begin{Pro} \label{hypoellipticity smooth - distribution}
     Let $P$ be a constant coefficient differential operator in $\RN$. If  $P$ is $\hypo( G^t, G^s)$ for $t>s$, then $P$ is   $\hypo( \D'_s, G^s)$.
  \end{Pro}

\begin{proof}
  Assume that $\hypo( \D'_s, G^{s})$ does not hold; then there exist an open set $U\subset \RN$ and $u \in \D'_s(U)\setminus G^{s}(U)$ such that $Pu=f \in G^s(U)$. Let $x_0 \in U$ such that the germ of $u$ at $x_0$ cannot be represented as a Gevrey function of order $s$ in any neighborhood of $x_0$.
 Fix $V\subset \subset U$ an open neighborhood of $x_0$.  By Theorem~\ref{prop a3}, there exist  a strongly elliptic ultradifferential $Q$ operator of type $\{p!^s\}$,   $g \in G^{s}(V)$ and $v\in G^t(V)\setminus G^s(V)$ such that
 \begin{equation*}
 \text{$Qv= u$ in $\D'_s(V)$ and $Qg =f$ in $G^s(V)$}.
 \end{equation*}  
Then one has, on $V$, 
\begin{equation*}
QPv= PQv = Pu = f= Q g \ \ \Longrightarrow \ \ Q(P v- g) =0.
\end{equation*}
Thanks to Theorem~\ref{propdokernel}, there exists $h\in \Com(V)$ such that $Pv= g+h$, proving that $Pv \in G^{s}(V)$ and $P$ is not $\hypo(G^t, G^{s})$, finalizing the proof.
\end{proof}

\begin{Cor}
  Let $P$ be a constant coefficient differential operator in $\RN$.
If  $P$ is not $\hypo(\D', G^s)$, then, for every $t>s$, there exist $U$ an open set and $u \in G^t(U)\setminus G^{s}(U)$ such that $Pu \in G^{s}(U)$. 
 \end{Cor}

The next step is to prove an equivalence between statements \eqref{eqs3} and \eqref{eqs4}.

 \begin{Pro} \label{equivalence 3-4}
   Let $P$ be a constant coefficient differential operator in $\RN$. Then $P$ is $\hypo(\D'_s, \D')$ if and only if $\hypo(\Cinf, G^{s})$.
 \end{Pro}
 \begin{proof}
   If $P$ is not $\hypo( \Cinf, G^s)$ there exist an open set $U \subset \RN$  and $u \in \Cinf(U)\setminus G^{s}(U)$ such that $Pu \in G^{s}(U)$.
Let $x_0$ be a point in $U$ such that the germ of $u$ at $x_0$ cannot be represented by a Gevrey function of order $s$ in any neighborhood of $x_0$.
   By Theorem~\ref{PropA1}, one can find an ultradifferential operator $Q$ of type $\{p!^s\}$ and  $V \subset \subset U$ an open neighborhood of $x_0$ such that  $Qu|_V \in \D'_s(V)\setminus \D'(V).$ On the other hand $P(Qu) = Q(Pu)\in G^{s}(V)$, which shows that  $P$ is not $\hypo(\D'_s, \D').$

   Now assume that $P$ is $\hypo(\Cinf, G^s)$; by Proposition~\ref{hypoellipticity smooth - distribution}, $P$ is also $\hypo(\D'_s, G^s)$. Let $U \subset \RN$ be an open set and consider $u \in\D'_s(U)$ for which $Pu \in \D'(U)$. For any $x_{0}\in U$, take $ V$ an open neighborhood of $x_{0}$ and $\chi\in \Cinf_c(U)$ such that $\chi \equiv 1$ on $V$. 
   
   Consider $f= \chi P u\in \E'(U)$ and denote by $E$ the fundamental solution of $P$. Then $P(f\ast E) = f\ast (P E)=f.$ Thus, by defining $v= f \ast E \in \D'(U)$, it follows that $P(u-v)=0$ on $V$. Since $P$  is $\hypo(\D'_s, G^s)$, $u-v = g\in G^{s}(V)$, which implies that $u= v+g\in \D'(V)$. Hence $u$ is a distribution in a neighborhood of $x_{0}$ and since $x_{0}$ can be taken arbitrarily, we deduce that $u\in \D'(U)$, finalizing the proof. 
 \end{proof}

Finally we prove an equivalence between statements \eqref{eqs2} and \eqref{eqs5}.

 \begin{Pro} \label{equivalence 2-5}
    Let $P$ be a constant coefficient differential operator in $\RN$. Then $P$ is $\hypo(\D'_s, \Cinf)$ if and only if $P$ is $\hypo(\D', G^s)$.
  \end{Pro}
  
  \begin{proof}
 Suppose that $P$ is $\hypo(\D', G^{s})$. By Proposition~\ref{hypoellipticity smooth - distribution}, $P$ is $\hypo(\D'_s, G^{s})$  as well. Let $U \subset \RN$ be an open set and $u \in \D'_s(U)$ such that $Pu \in \Cinf(U)$. Fix $x_{0} \in U$ and take $\chi\in \Cinf_c(U)$ such that $\chi=1$ on  $V\subset\subset U$, where $V$ is an open neighborhood of $x_{0}$. 
 
 Set $v= E \ast (\chi Pu)\in \Cinf(U)$, where $E$ denotes the fundamental solution of $P$. It follows that $Pv= Pu$ on $V$, which implies, by the assumption that $P$ is $\hypo(\D'_s, G^{s})$,  that $u - v\in G^{s}(V)$. Therefore $u\in \Cinf(V)$; since $x_{0}$ is arbitrary, we infer that $P$ is $\hypo(\D'_s, \Cinf)$. 
 
 Now we intend to show that 
 \begin{equation*}
 \text{$P$ is $\hypo(\D'_s, \Cinf)  \ \Longrightarrow \  P$ is $\hypo(\D', G^{s}) \ \Longleftrightarrow \ P$ is $\hypo(\Cinf, G^{s})$,}
 \end{equation*}
where the second implication is a consequence of Proposition~\ref{hypoellipticity smooth - distribution}. However, if $P$ is \emph{not} $\hypo(\Cinf, G^{s})$, by the argument applied in Proposition~\ref{equivalence 3-4}  one can find an open set $V \subset \RN$ and
\begin{equation*}
v \in \D'_s(V)\setminus \D'(V) \ \text{such that} \ Pv \in G^{s}(V) \subset \Cinf(V),
\end{equation*}
which proves that $P$ is not $\hypo(\D'_s, \Cinf)$.
 \end{proof}

 
\begin{proof}[Proof of Theorem \ref{teo1}]
	It is immediate that \eqref{eqs1} $\Rightarrow$ \eqref{eqs2} $\Rightarrow$ \eqref{eqs3}$\Rightarrow$ \eqref{eqs3'}. By Proposition~\ref{hypoellipticity smooth - distribution}, we have \eqref{eqs3'} $\Rightarrow$ \eqref{eqs1}, which shows the equivalence of the four first assertions. then it follows from Proposition~\ref{equivalence 3-4} and Proposition~\ref{equivalence 2-5} that \eqref{eqs3} $\Leftrightarrow$ \eqref{eqs4} and \eqref{eqs2} $\Leftrightarrow$ \eqref{eqs5}, which finalizes the proof. \end{proof}

\begin{Obs} \label{closed set hypo}
	When a differential operator with constant coefficients $P$ is $\hypo(\D', \Cinf)$, there exists $s_0\in [1, \infty)$ for which
\begin{equation} \label{hormander Gevrey}
\text{$P$ is $\hypo(\D', G^{s})$ if and only if $s \geq s_0$ (see \cite[Section $2.2$]{rod}).}
\end{equation} 
On the other hand,  if $P$ is $\hypo(\D', G^{t_0})$ for some $t_0 \in [1,\infty)$, it is $\hypo(\D', \Cinf)$. Theorem~\ref{teo1} proves in particular that if $P$ is $\hypo(\D', \Cinf)$, it is also $\hypo(\D_{t_0}', \Cinf)$, for some $t_{0} > 1$, while Theorem~\ref{Theorem 1} shows that this holds for hyperfunctions if and only if $P$ is elliptic. The threshold is given precisely by the constant $s_{0}$ described in \eqref{hormander Gevrey}.  
\end{Obs}

\section{Hypoellipticity for operators which commute with an elliptic differential operator in the smooth context} \label{section4}

For the remainder of the section, we consider $\Omega \subset \R^{n}$ an open set and any differential operator mentioned will have smooth coefficients on $\Omega$.

\begin{Teo} \label{hypo sobolev}
Let $P$ be a differential operator which commutes with  an elliptic differential operator $J$ of order $m \in \N$ and let $r,  t \in \R$ such that $r < t$. Then the following properties hold:
\begin{enumerate}
	\item \label{pure sobolev} If  $P$ is  $\hypo(H_{\loc}^{r}, H_{\loc}^{t})$, then $P$ is  $\hypo(H_{\loc}^{r + km}, H_{\loc}^{t+km})$, for any $k \in \Z$.
	\item \label{property triple hypo} If  $P$ is  $\hypo(H_{\loc}^{r}, H_{\loc}^{t}, \Cinf)$, then $P$ is  $\hypo(H_{\loc}^{r + km}, H_{\loc}^{t+km}, \Cinf)$, for any $k \in \Z$.
\end{enumerate}

\begin{proof}
Let $U$ be an arbitrary open subset of $\Omega$; we start the proof of  \eqref{pure sobolev} when $k$ is a positive integer. If $u \in H_{\loc}^{r + km}(U)$ and $Pu$ is $ H_{\loc}^{t+km}(U)$ then $J^{k}u \in H_{\loc}^{r}(U)$ and  $J^{k}Pu = P(J^{k} u) \in H_{\loc}^{t}(U)$. By hypothesis  $J^{k} u \in H_{\loc}^{t}(U)$; since $J^{k}$ is elliptic, it is a consequence of \cite[Theorem~18.1.29]{hor3} that $u \in H_{\loc}^{t+km}(U)$, which proves statement \eqref{pure sobolev} for positive $k$. 

Now suppose that $k$ is a fixed negative integer; then there exists (see for instance \cite[Corollaries $4.3$ and $4.4$]{tre4}) a properly supported elliptic pseudodifferential operator $E_{k}$ of order $km$   such that $J^{-k} E_{k} = I + R$, where $R$ is a regularizing pseudodifferential operator. Take $u \in H_{\loc}^{r + km}(U)$ such that $Pu$ is $ H_{\loc}^{t+km}(U)$; note that $E_{k}u \in H_{\loc}^{r}(U)$. Moreover, 
\begin{align*}
J^{-k} P(E_{k}u ) =  P(J^{-k} E_{k}u )& = Pu + P(Ru).
\end{align*}
Since $Pu \in H_{\loc}^{t+km}(U)$ and $ P(Ru)\in \Cinf(U)$
it follows that $J^{-k} P(E_{k}u ) \in H_{\loc}^{t+km}(U)$ and $P(E_{k}u ) \in H_{\loc}^{t}(U)$. By hypothesis  $E_{k}u \in H_{\loc}^{t}(U)$; since $E_{k}$ is elliptic as well, we deduce that $u \in H_{\loc}^{t + km}(U)$, as we intended to prove. 

Now we proceed to the proof of \eqref{property triple hypo}; let $k$ be a natural number and $u \in H_{\loc}^{r + km}(U)$ such that $Pu$ is $\Cinf(U)$; then $J^{k}u \in H_{\loc}^{r}(U)$ and  $J^{k}Pu = P(J^{k} u) \in \Cinf(U)$. Since $P$ is  $\hypo(H_{\loc}^{r}, H_{\loc}^{t}, \Cinf)$, one has  $J^{k} u \in H_{\loc}^{t}(U)$. Since $J^k$ is elliptic, $ u \in H_{\loc}^{t+km}(U)$, proving the that $P$ is  $\hypo(H_{\loc}^{r+km}, H_{\loc}^{t+km}, \Cinf)$ when $k$ is positive. 

Next we fix a negative integer $k$, consider once again $E_{k}$ a properly supported elliptic pseudodifferential operator of order $km$ such that $J^{-k} E_{k} = I + R$. If $u \in H_{\loc}^{r + km}(U)$ and $Pu$ is $\Cinf(U)$, then $E_{k}u \in H_{\loc}^{r}(U)$ and $P(E_{k} u) \in \Cinf(U)$. By hypothesis $E_k u \in H_{\loc}^{t}(U)$ and thus $ u \in H_{\loc}^{t + km}(U)$, which finalizes the proof.  
\end{proof}

\end{Teo}

\begin{Cor}
Let $P$ be a differential operator which commutes with  an elliptic differential operator of order $m$. Assume that $P$ is  $\hypo(H_{\loc}^{r}, H_{\loc}^{t}, \Cinf)$ and $t \geq r + m$. Then $P$ is $\hypo (\D', \Cinf)$.  
\end{Cor}
 
\begin{proof}
  Let $U\subset \Omega$ be an open subset and let $u \in \D'(U)$ such that $Pu \in \Cinf(U)$. If $V\subset\subset U$ is a relatively compact open subset, there exists $k \in \N$ for which $u \in H_{\loc}^{r - km}(V)$. Theorem~\ref{hypo sobolev} implies that $u \in H_{\loc}^{t - km}(V)$. Now observe that $u \in H_{\loc}^{r - k (m-1)}(V)$, since $t \geq r + m.$ We infer from Theorem~\ref{hypo sobolev} that $u \in H_{\loc}^{t - k (m-1)}(V)$. By successively applying this argument, we conclude that $u \in H_{\loc}^{r - km+ k\ell}(V)$ for any positive integer $\ell$, which implies that $ u \in \Cinf(V)$. The result follows from the fact that $V$ is arbitrary.
\end{proof} 

\begin{Cor} 
	Let $P$ be a differential operator which commutes with  an elliptic differential operator. If $P$ is  $\hypo(\D', H_{\loc}^{t}, \Cinf)$ for some $t \in \R$, then $P$ is $\hypo(\D', \Cinf)$.
\end{Cor}

\begin{Obs}
It is worth mentioning that if $P$ is $\hypo(H_{\loc}^{r}, H_{\loc}^{t})$, it is also $\hypo(H_{\loc}^{r}, H_{\loc}^{t}, \Cinf)$. Furthermore if $P$ is $\hypo (\D', \Cinf)$ it follows immediately that  $P$ is $\hypo(H_{\loc}^{r}, H_{\loc}^{r+m}, \Cinf)$, for every $r,m \in \R$.
\end{Obs}

\begin{Cor} \label{triple hypo I}
	Let $P$ be a differential operator which commutes with  an elliptic differential operator of order $m \in \N$. Then $P$ is $\hypo (\D', \Cinf)$ if and only if there exists $r \in \R$ for which $P$ is  $\hypo(H_{\loc}^{r}, H_{\loc}^{r+m}, \Cinf)$.  
\end{Cor}

\begin{Cor} \label{operators not hypoelliptic}
Let $P$ be a differential operator which commutes with  an elliptic differential operator of order $m \in \N$. If $P$ is not $\hypo (\D', \Cinf)$, then for every  $j \in \Z$ there exist an open subset $U_{j}\subset \Omega$ and $u_{j}\in H_{\textrm{loc}}^{j }(U_{j})\setminus H_{\loc}^{j+m}(U_{j})$ such that $Pu_{j}\in \Cinf(U_{j})$. In particular, we have the following consequences:
\begin{itemize}
	\item \emph{For every $r \in \R$, there exists an open subset $U_{r}\subset \Omega$ and $u \in \D'(U_{r}) \setminus H_{\loc}^{r}(U_{r})$ such that $Pu \in \Cinf(U_{r})$.}
	\item \emph{For any $n \in \N$, there exists an open subset $V_{n}\subset \Omega$ and $u \in H_{\loc}^{n}(V_{n}) \setminus \Cinf(V_{n})$ for which $Pu \in \Cinf(V_{n})$.  }
	\item  \emph{For each $k \in \N$, there exists an open subset  $W_{k}\subset \Omega$ and $u \in \mathcal{C}^{k}(W_{k}) \setminus \Cinf(W_{k})$ such that $Pu \in \Cinf(W_{k})$. }
\end{itemize}

\end{Cor}

A   constant coefficient operator on $\R$  commutes with the usual derivative while  on $\R^2$ it commutes with the Cauchy-Riemann operator. In $\R^{n}$, with $n \geq 3$, one needs to consider a second order  elliptic operator such as the Laplace operator. Therefore we have

 \begin{Cor}\label{triple hypo II}
Let $P$ be a constant coefficient operator on $\R^{n}$, $n \in  \N$. Then $P$ is $\hypo(\D', \Cinf)$  if and only if $P$ is $\hypo(L_{\textrm{loc}}^{2},H_{\textrm{loc}}^{2},  \Cinf)$.
\end{Cor}

We proceed now to important examples of operators whose coefficients are not constant. 

\begin{Exe}[\emph{Tube type} vector fields] \label{ttvf}
 On $\R^2$, one can always find an elliptic operator of order $2$ which commutes with a vector field of \emph{tube type}, i.e., a vector field with the following form
\begin{equation} \label{tube type}
    L= \frac{\del}{\del x}+  a(x) \frac{\del}{\del y},
\end{equation}
 where $a$ is a smooth function. In fact, given $x_0\in \R$, fix $R > 0$ and let $M>1$ be such that 
 
 \begin{equation*}
M \geq 2 \Big( \sup_{x \in I} |a(x)|^{2} \Big), \textrm{ where } I=[x_{0} - R, x_{0} + R].
 \end{equation*}
Then the operator  $Q = \displaystyle M\left(\frac{\del}{\del y}\right)^{2} + L^2$ is elliptic on $I\times \R$ and commutes with $L$. 
\end{Exe}

Example~\ref{ttvf} may be applied to important families of vector fields such as the following:

\begin{Exe}[The Mizohata's Operators] \label{Mizohata}
Consider the family of vector fields on $\R^{2}$ described by
\begin{equation*}
L_{k} = \frac{\del }{\del x} + ix^{k} \frac{\del }{\del y}, \ \ k \in \N.
\end{equation*}
It is immediate that $L_{k}$ has the form \eqref{tube type}. 
Furthermore $L_{k}$ is elliptic on any  open set that \emph{does not}  intersect $\{(0,y)\in \R^2\}$.  Otherwise its hypoellipticity depends on the parity of $k$. 

\begin{itemize}
	\item When $k$ is even  $L_{k}$ is $\hypo(\D', \Cinf)$, $\hypo(\D', G^{s})$ for $s > 1$ and  $\hypo(\D', \Com)$ (see \cite[pages 164/165]{rod2}). 
	\item When $k$ is odd, $L_{k}$ is not $\hypo (\D', \Cinf )$ (see \cite[Appendice]{miz}), so the statement on Corollary~\ref{operators not hypoelliptic}  holds.
\end{itemize}
\end{Exe}

\begin{Exe}[Okaji's Operators] \label{Okaji}
Consider the family of operators on $\R^{2}$ given by 
\begin{equation*}
P = \left( \frac{\del }{\del x} + ix^{2k} \frac{\del }{\del y}\right)^{2} + c \frac{\del }{\del y} = L_{2k}^{2} + c \frac{\del }{\del y} , 
\end{equation*}
where $k$ is a positive integer and $c$ is a non-zero complex number. 
It was proved in \cite[Theorem 2]{oka} that $P$ is not $\hypo(\D', \Cinf)$ at the origin. Let $Q$ be the operator constructed in Example~\ref{ttvf} associated to $L_{2k}$; it is immediate that $Q$ commutes with $P$. Hence  Corollary~\ref{operators not hypoelliptic} is applicable, where the open subsets taken are once again neighborhoods of the origin.
\end{Exe}

\begin{Exe} [Lewy's Example] \label{Lewy} Consider the example  presented in \cite{l}  of an operator that is not locally solvable (in the $(\D', \Cinf)$ sense) at the origin, given by
\begin{equation*}
L = \frac{1}{2} \left[\frac{\del }{\del x} + i  \frac{\del }{\del y}\right] +  (y - ix) \frac{\del }{\del w}, \ \ \text{with $(x, y, w) \in \R^{3}$.}
\end{equation*} 
Note that $\transp{L} = - L$, thus $\transp{L}$ is not locally solvable at the origin either. This implies in particular that $L$ is \emph{not} $\hypo(\D', \Cinf)$ at $0$ (see \cite[Chapter 52]{tre2}). 

On the other hand, $L$ is a vector field which spans the tangent bundle $\mathcal{V}$ of a \emph{CR structure} (see \cite[Example I.7.1]{tre3}), which is also \emph{locally integrable} (see \cite[Definition I.1.2]{tre3}) on any open set of $\R^{3}$. Next we follow the notation used in \cite[Section I.7]{tre3}: let
\begin{equation*}
M_{1} = \frac{\del}{\del x} - 2ix \frac{\del}{\del w}, \ \ M_{2} = \frac{\del}{\del w}.
\end{equation*}
Then $L$, $M_{1}$ and $M_{2}$ are pairwise commuting vector fields which form a $\Cinf$ basis of $\C T \R^{3}$  over an open neighborhood $U$ of $0$. Then it is possible to prove (see \cite[p.99]{tre3}) that 
\begin{equation} \label{elliptic lis}
P = L^{2} + \kappa \left(M_{1}^2 + M_{2}^{2} \right) \ \text{is \emph{elliptic} on $U$, for a sufficiently large $\kappa > 0$. }
\end{equation}
Since $P$ clearly commutes with $L$, it follows that Corollary~\ref{operators not hypoelliptic} is once again applicable.

\end{Exe}

\begin{Obs}
The argument used in Example $\ref{Lewy}$ can be extended for any vector field $X$ related to a locally integrable structure (i.e. it could be either an $L$ or an $M_{k}$, using the notation established in \cite{tre3}), assuming that $X$ is not $\hypo(\D', \Cinf)$.  
\end{Obs}

\section{Extending  hypoellipticity} \label{section5}

Let $\mathcal{F}, \mathcal{G}$  and $\mathcal{H}$ be three sheaves such that $\mathcal{H}\subset \mathcal{G}\subset \mathcal{F}$. If an operator is $\hypo(\mathcal{G},\mathcal{H})$, it is natural to wonder whether it is also $\hypo(\mathcal{F},\mathcal{H})$. If so, we shall say that the hypoellipticity $\hypo(\mathcal{G},\mathcal{H})$ \emph{extends} to $\mathcal{F}.$ In this section, we intend to study models of operators for which there exists (or not) such an extension property.

\subsection{An one-dimensional model}

We consider  operators defined over an open  interval $I\subset \R$ given by
 \begin{align*}
   P= \sum_{j=0}^{m} a_j(x) \frac{\dd^{j}}{\dd x^{j}}, \ \ \text{where each $a_j\in \Com(I)$.}
 \end{align*}
Now we recall a definition set in \cite{kom2}: let $x_{0} \in I$ be a zero of $a_{m}$; the \emph{irregularity} of $P$ at $x_{0}$ is given by the number
\begin{equation} \label{definition irregularity}
  \sigma= \max \bigg\{ 1, \max_{0 \leq i < m} \frac{ \ord_{x_{0}} a_{m}- \ord_{x_{0}} a_{i}}{m-i} \bigg\},
\end{equation}
where $\ord_{x_{0}}a_i$ denotes the order of vanishing of $a_i$ at $x_{0}$ (note that it may be zero when $a_{i}$ does not vanish at the point).

Komatsu proved in \cite[Theorems E and F]{kom2} the following results:
\begin{equation} \label{result  Komatsu}
\begin{split}
&\text{When $\sigma = 1$, $P$ is $\hypo(\mathcal{B}, \D')$}; \\
&\text{If $\sigma \leq \displaystyle \frac{s}{s-1}$,  $P$ is $\hypo(\mathcal{B}, \D'_{(s)})$.}
\end{split} 
\end{equation}  
Furthermore, Cordaro and Trépreau proved in \cite[Lemma 2.1]{ct} that
\begin{equation}\label{Cordaro-Trepreau result}
\text{for every $ \sigma < \displaystyle  \frac{s}{s-1}$, $P$ is $\hypo(G^{(s)}, \Com)$.}
\end{equation}

For instance, consider the particular case where
\begin{align*}
  P_k= x^{k+1}\frac{\dd}{\dd x} - k, \ \ \text{for some $k \in \N$}. 
\end{align*}
Note that the function $x \mapsto x^{k+1}$ only vanishes at the origin, hence $P_k$ has irregularity  $\sigma= k+1$ at $0$. It is then a consequence of \eqref{Cordaro-Trepreau result} that  $P_k$ is $\hypo(G^{(s)}, \Com)$ for every $s\in \left(1, 1+\frac{1}{k}\right)$. In particular, due to the fact that $G^{t}\subset G^{(s)}$ for any $t<s$, we conclude that 
\begin{equation*}
P_k\text{ is $\hypo(G^{t}, \Com)$ for every $t\in \Big(1, 1+\frac{1}{k} \Big).$}
\end{equation*}
On the other hand, $P_k f=0$ has a solution
\begin{align*}
  f(x)=\left\{
  \begin{array}{cl}
    e^{-1/x^{k}},& \textrm{ if } x > 0;\\
    0,& \textrm{ if } x\leq 0
  \end{array}\right.
\end{align*}
that is clearly not real-analytic and belongs to $G^{1+ \frac{1}{k}}(\R).$ Therefore $P_k$ is not $\hypo(G^{s}, \Com)$ for any $s\geq 1+ \frac{1}{k}$.
Hence, for this class of operators, the hypoellipticity $\hypo(G^{s}, \Com)$ is completely determined by its irregularity and cannot be extended.

\subsection{Operators of tube type in product spaces} \label{tube product}

Let be $V\times U\subset \R^n_t\times \R^m_x$ an open set and consider $P$ a differential operator in $V\times U$ of the form 
\begin{align}\label{PTubeCase}
  P = \sum_{|\alpha|\leq k} \sum_{|\beta|\leq \ell} a_{\alpha \beta}(t) \del_t^{\alpha} \del_x^{\beta},
\end{align}
where each $a_{\alpha \beta}$ is a real-analytic function. Denote
\begin{align}\label{P0TubeCase}
  P_0 = \sum_{|\alpha|\leq k}  a_{\alpha 0}(t) \del_t^{\alpha};
\end{align}
that is, the operator associated when one takes $\beta = 0$, which can be seen as an operator acting on sheaves on $V$. 

Assume that $\mathcal{F}$ and $\mathcal{G}$ are two sheaves for which $P$ induces a morphism. Note that
\begin{align*}
  P(\varphi \otimes 1)= (P_0 \varphi) \otimes 1\in \mathcal{F}(V\times U), \quad \forall \varphi \in \mathcal{F}(V).
\end{align*}
Therefore if $P$ is $\hypo(\mathcal{F},\mathcal{G})$ on $V\times U$, $P_0$ is $\hypo(\mathcal{F},\mathcal{G})$ on $V$. Indeed, if $\varphi \in \mathcal{F}(V_1)$ for some $V_1\subset V$ is such that $P_0 \varphi \in \mathcal{G}(V_1)$, we have
\begin{align*}
 (P_0 \varphi) \otimes 1=   P(\varphi \otimes 1)\in \mathcal{G}(V_1\times U).
\end{align*}
Since $P$ is $\hypo(\mathcal{F},\mathcal{G})$, it follows that $\varphi\otimes 1 \in \mathcal{G}(V_1\times U)$ and therefore $\varphi \in \mathcal{G}(V_1)$.

This discussion shows that the hypoellipticity of $P_0$ is necessary for $P$ to have the same property. We will work under a stronger assumption on $P_0$:
\begin{itemize}
	\item $P_0$ is elliptic on $V$;
	\item $P$ and $P_0$ have the same order. 
\end{itemize} 
We intend to prove that if $P$ is $\hypo(G^{r}, G^{s})$, for some  $r>s$, it is also $\hypo(\D'_s, G^{s})$. That is, the $G^{s}$-hypoellipticity is extendable to ultradistributions of order $s$.

Recall that the  principal symbol of $P$ is given by
\begin{align*}
  p_{m}(t, x, \tau, \xi)= \sum_{k+\ell=m} \sum_{|\alpha|= k} \sum_{|\beta|= \ell} (i)^{|\alpha| + |\beta|} a_{\alpha \beta}(t) \tau^{\alpha} \xi^{\beta}.
\end{align*}
If we denote by $(p_0)_m$ the principal symbol of $P_0$, then
\begin{align*}
  p_{m}(t, x,  \tau, 0)= (i)^{m} \sum_{|\alpha|= m}  a_{\alpha 0}(t) \tau^{\alpha} = (p_0)_m(t, \tau).
\end{align*}
Since $P_{0}$ is elliptic, it follows from the identity above that
\begin{equation} \label{characteristic subset}
\begin{split}
 \Char~P&= \{(t, x, \tau, \xi)\in V\times U \times (\R^{n+m}\setminus\{0\}): p_{m}(t,x, \tau, \xi)=0\}\\
&\subset \{(t, x, \tau, \xi)\in V\times U \times \R^{n+m}: \xi\neq 0 \}.
\end{split}
\end{equation}

Now take $u\in \E'_s(V\times U)$ such that $Pu \in G^{s}(V \times U)$;  thanks to the elliptic regularity theorem, we have
  \begin{align*}
  WF_s(u) \cap \{(t, x, \tau, 0)\in V\times U \times \R^{n+m} \}=\emptyset.
\end{align*}
This means that for every $(t_0,x_0) \in \supp~u$, there exist a neighborhood  $W_{t_0,x_0}$ and  an open cone $\Gamma_{t_0,x_0}\subset\R^{n+m}$ such that $\{(\tau, 0): \tau \in \R^n\setminus\{ 0\}\} \subset \Gamma_{t_0,x_0}$, for which
\begin{align*}
 WF_s(u) \cap (W_{t_0,x_0}\times \Gamma_{t_0,x_0})=\emptyset.
 \end{align*}
Since the support of $u$ is compact, one can find an open cone $\Gamma$ with sufficiently small aperture such that $\{(\tau, 0): \tau \in \R^n\setminus\{ 0\}\} \subset \Gamma$ and
\begin{align}\label{wavefrontsetlongdegamma}
 WF_s(u) \cap (V\times U \times \Gamma)=\emptyset.
 \end{align}
 In particular we may take $\kappa\in (0,1)$ such that
 \begin{align*}
 \Gamma=  \{(\tau, \xi)\in \R^{n+m}: |\tau|> \kappa|\xi|\}.   
 \end{align*}

\begin{Teo} \label{Gevrey hypo implies ultra hypo}
 Consider $r>s>1$ and assume $P$ as described above. If $P$ is $\hypo(G^{r}, G^{s})$, then it is $\hypo(\D'_s, G^{s})$.  
\end{Teo}
\begin{proof}
Note first that if $P$ is $\hypo(G^{r_0}, G^{s})$ it is also $\hypo(G^{r}, G^{s})$, for any  $s<r<r_0$. Hence one may assume that $s< r< 2s$, which allows us to write 
\begin{equation} \label{expression for mu}
r= \frac{s}{(1-\mu)},\textrm{ for some } \mu \in (0, 1/2 ).  
\end{equation}
Let $u\in \D'_s(V\times U)$ such that $Pu\in G^{s}(V\times U)$; we shall prove that any $(t_0, x_0) \in V\times U$ has a neighborhood $V_0\times U_0$ on which $u \big{|}_{V_0\times U_0} \in G^{s}(V_0 \times U_0)$.
 
Indeed, let $V_0\subset\subset V$ and $U_0\subset \subset U$ be any relatively compact neighborhoods of $t_0$ and $x_0$, respectively. Take $\chi\in G^{s}_c(V\times U)$ such that $\chi \equiv 1$ on $V_0\times U_0$; then $\chi u \in \E'_s(V \times U)$ and for every $\epsilon>0$ there exists  a constant $C_\epsilon>0$ such that
\begin{align}\label{Desultradistribuicao}
  |\widehat{\chi u}(\tau, \xi)|\leq C_\epsilon e^{\epsilon |(\tau,\xi)|^{1/s}}, \quad \forall (\tau,\xi) \in \R^{n+m},
\end{align}
where $\widehat{\chi u}$ represents the Fourier transform of $\chi u$.

Let $c_{1}: (0, \infty) \to  (0, \infty)$ be the non-increasing function given by 
\begin{equation*}
c_{1} (\epsilon) = \inf \left\{M > 0; \ \ |\widehat{\chi u}(\tau, \xi)|\leq M e^{\epsilon |(\tau,\xi)|^{1/s}}, \ \forall (\tau,\xi) \in \R^{n+m} \right\}.
\end{equation*}  
Take a continuous and non-increasing function $c_{2}: (0, \infty) \to  (0, \infty)$  such that 
\begin{equation*}
\text{$c_{2} (\epsilon) \geq c_{1}(\epsilon)$, for every $\epsilon > 0$.}
\end{equation*}
 By setting $C:(0, \infty) \lra (0, \infty)$ as
\begin{equation*}
 C(\epsilon) = c_{2}(\epsilon) + \displaystyle \frac{1}{\epsilon},
\end{equation*}
we have the following facts:
\begin{itemize}
	\item $C$ is strictly decreasing;
	\item $C(\epsilon)> 1$ if $\epsilon\leq 1$;
	\item $C(\epsilon) \lra \infty$ when $\epsilon \lra 0^{+}$; 
	\item the constant $C_{\epsilon}$ in \eqref{Desultradistribuicao} can be taken as $C(\epsilon)$, for every $\epsilon \in (0, +\infty)$. 
\end{itemize}  

 Let $a:= \log C(1)>0$ and consider $\log C: (0,1]\lra [a, \infty)$. Since $\log C$ is strictly decreasing, it is surjective. Therefore, we can take $\gamma: [a, \infty) \lra (0,1]$ to be its inverse, which is strictly decreasing as well, and define $\sigma: (0,\infty) \lra (0, \infty)$ by
\begin{align*}
 \sigma(\rho)=
  \left\{\begin{array}{cl}
          \displaystyle \frac{\log 2}{6} \displaystyle\frac{1}{ \rho^{-\mu}+ \gamma(\rho^{\mu})}, &\textrm{if }  \rho^{\mu} \geq a;\\
           \displaystyle\frac{\log 2}{6} \displaystyle \frac{1}{ a^{-\mu}+ \gamma(a^{\mu})}, &\textrm{if } 0< \rho^{\mu}< a,
         \end{array}
      \right.
\end{align*}
where $\mu$ is the constant defined in \eqref{expression for mu}. Observe that $\sigma$ is a continuous and increasing function satisfying
\begin{align*}
\lim_{\rho \lra \infty} \sigma(\rho) = \infty.  
\end{align*}

Now we can proceed as in the Appendix \ref{Apendice 1} and use $\sigma$ to define the symbol of an ultradifferential operator as
\begin{align*}
  Q(\zeta) = \prod_{p=1}^{\infty} \bigg( 1- \frac{\left\langle \zeta \right\rangle^{2}}{p^{2s} \sigma(p)}\bigg),
\end{align*}
where $\left\langle \zeta \right\rangle^{2}= \zeta_{1}^{2} + \cdots + \zeta_{m}^{2}$ is the holomorphic extension of the norm.
By Lemma~\ref{Operadoresdeordeminfinitaexistem}, $Q$ is an ultradifferential operator of type $\{p!^s\}$ which is strongly elliptic. That is,
\begin{equation} \label{elliptic cone}
|Q(i \zeta)| \geq 1 \ \text{if  $\zeta= \xi + i \eta$ and $|\eta| \leq \frac{|\xi|}{2}$.}
\end{equation}

For each fixed $\xi \in \R^m$, take a positive integer $q_{\xi}$ such that
\begin{equation} \label{intermediate inequality}
q_{\xi}^{2s} \sigma(q_{\xi}) \leq |\xi|^{2} \leq (q_{\xi}+1)^{2s} \sigma(q_{\xi} + 1).
\end{equation}
If $|\xi|$ is sufficiently large, one has $\sigma(q_{\xi}) \geq 1$ and so 
\begin{equation} \label{consequences intermediate inequality}
 \ q_{\xi} \leq |\xi|^{1/s} \textrm{ and } q_{\xi} + 1 \geq \frac{|\xi|^{1/s}}{\sigma(q_{\xi} + 1)^{1/2s}}\geq \frac{|\xi|^{1/s}}{\sigma(|\xi|^{1/s} + 1)^{1/2s}} \geq \frac{|\xi|^{1/s}}{\sigma(|\xi|^{1/s} + 1)}.
\end{equation} 
Therefore, when $|\xi|$ is sufficiently large, we deduce from  \eqref{intermediate inequality} and \eqref{consequences intermediate inequality} that
\begin{align} \label{estimate for Q}
  |Q(i \xi)|\geq \prod_{p=1}^{q_\xi}\bigg| 1+ \frac{|\xi|^{2}}{ p^{2s} \sigma(p)}\bigg| \geq 2^{q_{\xi}} \geq \frac{1}{2} \exp \left[(\log 2) \left(\frac{|\xi|^{1/s}}{\sigma(|\xi|^{1/s} + 1)} \right) \right].
\end{align}

Now consider $A_\kappa=\{(\tau, \xi) \in \R^{n+m}: |\tau|\leq \kappa|\xi|\}$, for some $\kappa\in (0,1/2)$ such that the wave-front set of $\chi u$ satisfies \eqref{wavefrontsetlongdegamma}. Then $|(\tau, \xi)|^{2}= |\tau|^{2}+ |\xi|^{2}\leq (1+\kappa^{2})|\xi|^{2}$, which implies that 
\begin{equation} \label{comparison tau and xi}
|(\tau, \xi)|^{1/s}\leq \sqrt[2s]{1+\kappa^2}|\xi|^{1/s} \leq 2^{1/s} |\xi|^{1/s}.
\end{equation}
Therefore it follows from \eqref{Desultradistribuicao} and \eqref{estimate for Q} that if $|\xi|$ is sufficiently large, then
\begin{align*}
  \bigg| \frac{\widehat{\chi u}(\tau, \xi)}{Q(i\xi)}\bigg| &\leq  2C(\epsilon) \exp \bigg[\epsilon|(\tau, \xi)|^{1/s}-(\log 2) \bigg(\frac{|\xi|^{1/s}}{\sigma(|\xi|^{1/s}+1)} \bigg) \bigg]\\
  &\leq  2 \exp \bigg[\log C(\epsilon)+ \epsilon(1+\kappa^2)^{1/(2s)} |\xi|^{1/s}- (\log 2) \bigg(\frac{|\xi|^{1/s}}{\sigma(|\xi|^{1/s}+1)} \bigg) \bigg].
\end{align*}
Let $\epsilon= \gamma((|\xi|^{1/s}+1)^{\mu})$ and assume that $(|\xi|^{1/s}+1)^{\mu}> a= \log C(1)$. It is a consequence of the definitions of $\gamma$, $\sigma$, \eqref{comparison tau and xi} and last estimate that
\begin{align*}
  \bigg| \frac{\widehat{\chi u}(\tau, \xi)}{Q(i\xi)}\bigg| 
  &\leq  2 \exp\Big[(|\xi|^{1/s}+1)^{\mu}+ \gamma((|\xi|^{1/s}+1)^{\mu})(1+\kappa^2)^{1/(2s)} |\xi|^{1/s}   \\
  &- 6 |\xi|^{1/s}  \big( (|\xi|^{1/s}+1)^{-\mu}+ \gamma((|\xi|^{1/s}+1)^{\mu}) \big) \Big]\\
  &\leq  2 \exp\bigg((|\xi|^{1/s}+1)^{\mu}-  \frac{6 |\xi|^{1/s}}{(|\xi|^{1/s}+1)^{\mu} } \bigg)\\
  &\leq  2 \exp\bigg( 2|\xi|^{\mu/s}-  \frac{3 |\xi|^{1/s}}{|\xi|^{\mu/s} } \bigg).
\end{align*}
Since $\mu < \frac{1}{2}$ (see \eqref{expression for mu}), it follows from \eqref{comparison tau and xi} that
\begin{align*}
\bigg| \frac{\widehat{\chi u}(\tau, \xi)}{Q(i\xi)}\bigg|  &\leq  2 \exp\big(-|\xi|^{(1- \mu)/s}  \big)\\
&\leq  2 \exp\Big( - \frac{|(\tau,\xi)|^{1/r}}{2^{1/r}}  \Big)\\
&\leq  2 \exp\Big( - \frac{|(\tau,\xi)|^{1/r}}{2}  \Big), 
\end{align*}
for every $(\xi, \tau) \in A_{\kappa}$, assuming that $|\xi|$ is large enough. Hence there exists $C_A>0$ such that
\begin{align} \label{estimate A_kappa}
  \bigg| \frac{\widehat{\chi u}(\tau, \xi)}{Q(i\xi)}\bigg|     \leq  C_A e^{-  \frac{1}{2}|(\tau,\xi)|^{1/r} }, \quad \forall ( \tau, \xi) \in A_\kappa.
\end{align}

Next we define, at least formally, the function
\begin{align} \label{formula for v}
  v(t, x) = \frac{1}{(2\pi)^{n+m}} \iint  \frac{\widehat{\chi u}(\tau, \xi)}{Q(i\xi)} e^{i(\tau t+ \xi x)} \dd \tau \dd \xi.
\end{align}
Let $B_{\kappa}= \{(\tau, \xi)\in \R^{n+m}: |\tau|> \kappa|\xi|\}$; as a consequence of \eqref{wavefrontsetlongdegamma}, there exist $h>0$ and $C_B>0$ for which
\begin{align} \label{estimate B_kappa}
  |\widehat{\chi u}(\tau, \xi)|\leq C_B e^{-h|(\tau, \xi)|^{1/s}}, \quad \forall (\tau,\xi) \in B_{\kappa}.
\end{align}
Note that
\begin{align*}
  |e^{i(\tau t_1+ \xi x_1)}-e^{i(\tau t_2+ \xi x_2)}|
     &\leq 2\epsilon \big(|\tau|+ |\xi|\big), 
\end{align*}
whenever $|t_1-t_2|\leq \epsilon$ and $|x_1-x_2|\leq \epsilon$.
 Thus
\begin{align}
  \big|\del_t^{\alpha}\del_x^{\beta}  v(t_1, x_1)-\del_t^{\alpha}\del_x^{\beta}  v(t_2, x_2)\big| &\leq \frac{2\epsilon}{(2\pi)^{n+m}} \iint  \bigg|\frac{\widehat{\chi u}(\tau, \xi)}{Q(i\xi)}\bigg| |\tau^{\alpha}||\xi^{\beta}|(|\tau|+|\xi|) \dd \tau \dd \xi \nonumber \\
  &\leq \frac{2\epsilon}{(2\pi)^{n+m}}( I_{A_\kappa}+ I_{B_\kappa}), \label{splitting integral}
\end{align}
where $I_{A_{\kappa}}$ and $I_{B_\kappa}$ are the respective integrals of
\begin{align*}
  \bigg|\frac{\widehat{\chi u}(\tau, \xi)}{Q(i\xi)}\bigg| |\tau^{\alpha}||\xi^{\beta}|(|\tau|+|\xi|).
\end{align*}

%

Using \eqref{estimate A_kappa}, we deduce that
\begin{align}
I_{A_\kappa} &\leq  \iint_{A_\kappa}  C_A e^{- \frac{1}{2} |(\tau,\xi)|^{1/r} } |\tau^{\alpha}||\xi^{\beta}|(|\tau|+|\xi|) \dd \tau \dd \xi \nonumber\\
&\leq  \iint_{A_\kappa}  C_A e^{- \frac{1}{4} |(\tau,\xi)|^{1/r} } \Big(e^{- \frac{1}{4} |(\tau,\xi)|^{1/r} } |\tau^{\alpha}||\xi^{\beta}|\Big)(|\tau|+|\xi|) \dd \tau \dd \xi. \label{integral a_k}
\end{align}
Analogously, we conclude from \eqref{elliptic cone} and \eqref{estimate B_kappa} that
\begin{align}
I_{B_\kappa} & \leq \iint_{B_\kappa}  C_B e^{-  \frac{h}{2}|(\tau,\xi)|^{1/s} }\Big( e^{-  \frac{h}{2}|(\tau,\xi)|^{1/s} }|\tau^{\alpha}||\xi^{\beta}|\Big)(|\tau|+|\xi|) \dd \tau \dd \xi. \label{integral b_k}
\end{align}
Since for every $t>0$ we have that 
\begin{align*}
   e^{-s t^{1/s}}\leq \frac{j!^{s}}{t^{j}}, \quad \forall j \in \Z_+,
\end{align*}
it follows that in $B_\kappa$ we have the estimate
\begin{align}
  e^{-  \frac{h}{2}|(\tau,\xi)|^{1/s} }|\tau^{\alpha}||\xi^{\beta}|
    &= e^{- s \big(\frac{h^{s}}{(2s)^{s}}|(\tau,\xi)|\big)^{1/s}} |\tau^{\alpha}||\xi^{\beta}| \nonumber \\
    &\leq \frac{(|\alpha|+|\beta|)!^{s}}{ \big(\frac{h^{s}}{(2s)^{s}}|(\tau,\xi)|\big)^{|\alpha|+|\beta|}} |\tau|^{|\alpha|}|\xi|^{|\beta|} \nonumber \\
    &\leq \widetilde{h}^{|\alpha|+|\beta|} (|\alpha|+|\beta|)!^{s}, \label{estimate exponential Gevrey} 
\end{align}
where $\widetilde{h} = \frac{(2s)^{s}}{h^{s}}$. By possibly increasing  $\widetilde{h}$, we similarly obtain  
\begin{align} \label{estimate exponential Gevrey II}
  e^{-  \frac{1}{4}|(\tau,\xi)|^{1/r} }|\tau^{\alpha}||\xi^{\beta}|
      &\leq \widetilde{h}^{|\alpha|+|\beta|} (|\alpha|+|\beta|)!^{r}. 
\end{align}

By associating \eqref{integral a_k} to \eqref{estimate exponential Gevrey II} and \eqref{integral b_k} to \eqref{estimate exponential Gevrey}, we infer that
\begin{align*}
  I_{A_\kappa} & \leq   C_A \tilde{h}^{|\alpha|+|\beta|} (|\alpha|+|\beta|)!^{r} \iint_{A_\kappa} e^{- \frac{1}{4} |(\tau,\xi)|^{1/r} }(|\tau|+|\xi|) \dd \tau \dd \xi \leq   C_{1} \tilde{h}^{|\alpha|+|\beta|} (|\alpha|+|\beta|)!^{r} 
\end{align*}
and
\begin{align*}
  I_{B_\kappa}  & \leq   C_B \tilde{h}^{|\alpha|+|\beta|} (|\alpha|+|\beta|)!^{s} \iint_{A_\kappa} e^{- \frac{h}{2} |(\tau,\xi)|^{1/s} }(|\tau|+|\xi|) \dd \tau \dd \xi \leq   C_{1}\tilde{h}^{|\alpha|+|\beta|} (|\alpha|+|\beta|)!^{s},
\end{align*}
for some $C_{1} > 0$. This proves that $v$ and all its derivatives are continuous,  moreover, $v$ is a function of class $G^{r}$; furthermore, it follows from its formula that
\begin{align*}
  Q(P v)= P(Qv) = Pu \in G^{s}( V_0\times U_0).
\end{align*}
Define then
\begin{align*}
  w(t, x) = \frac{1}{(2\pi)^{n+m}} \iint  \frac{\widehat{\chi P u}(\tau, \xi)}{Q(i\xi)} e^{i(\tau t+ \xi x)} \dd \tau \dd \xi.
\end{align*}
With an analogous computation, it is not difficult to check  that $w\in G^{s}(V\times U)$ and $Qw= \chi Pu$, which implies that 
\begin{align*}
  Q(P v- w)=0 \ \textrm{on} \ V_0 \times U_0.
\end{align*}

Note that $Q$ remains a strongly elliptic ultradifferential operator of type $\{p!^s\}$  if one considers its action on open sets of $\R^{n+m}$ (see \eqref{elliptic cone} and Definition \ref{Strongly elliptic}). Hence it follows from Proposition~\ref{propdokernel} the existence of a neighborhood of $(t_{0}, x_{0})$ such that the restriction of any element of its kernel is a real-analytic function.

Thence there exist a neighborhood $V_1\times U_1$ and $f\in \Com(V_1 \times U_1)$ such that 
$$Pv= f+w\in G^{s}(V_1\times U_1).$$ 
Since $P$ is $\hypo(G^{r},G^{s})$,  it follows that $v\in G^{s}(V_1\times U_1)$.  Due to the fact that $Q$ is continuous on $G^{s}(V_1\times U_1)$, we obtain  $u=Q v\in G^{s}(V_1\times U_1)$, which ends the proof.
\end{proof}

\begin{Obs}[\textbf{Real-Analytic/Hyperfunctions Case}] \label{RA case}
Using hyperdifferential operators it is possible to adapt the previous result to the context of hyperfunctions proving that if $P$ has real-analytic coefficients, $P_0$ is elliptic and has the same order of $P$, then if $P$ is $\hypo(G^r, \Com)$ for some $r>1$, then it is $\hypo(\mathcal{B}, \Com)$.
\end{Obs}

\begin{Exe} \label{Oleinik operators}
Consider  the family of Oleinik operators studied in \cite{ole}, \cite{bt} and \cite{c}, given by
\begin{equation} \label{Lpq operator}
L_{p,q} = \del_{t}^2 + t^{2(p-1)} \del_{x_{1}}^{2} +  t^{2(q-1)} \del_{x_{2}}^{2}, \ \ \ (t, x_{1}, x_{2}) \in \R^{3},
\end{equation}
where $p, q \in \N$ and $p \geq q$. It is immediate that $L_{p,q}$ belongs to the class defined in Subsection \ref{tube product}.  It was proved in   \cite{bt} and \cite{c} that 
\begin{equation*}
L_{p,q} \ \text{is} \ \hypo (\D', G^{s}) \ \text{\emph{is and only if}} \ s \geq \frac{p}{q}.
\end{equation*}
In this context, Theorem \ref{Gevrey hypo implies ultra hypo} and Remark \ref{RA case} provide two different types of conclusions. On the one hand, we actually have 
\begin{equation} \label{oleinik property}
s \geq \frac{p}{q} \ \Longrightarrow \  L_{p,q} \ \text{is} \ \hypo (\D_{s}', G^{s}).
\end{equation}
On the other hand, fix $s < \frac{p}{q}$; for any given $r > s$ there exists an open set $U \subset  \R^{3}$ (depending on $r$) containing the origin and
\begin{equation} \label{oleinik property II}
 g \in G^{r}(U) \setminus G^{s}(U) \textrm{ such that } L_{p,q} g \in G^{s}(U).
\end{equation}

It is worth mentioning that the celebrated \emph{Baouendi-Goulaouic operator} is the particular case where $p = 2$ and $q=1$ (i.e. the operator $L_{2,1}$). In this case we obtain:

\begin{itemize} 
\item $L_{2,1}$ is $\hypo (\D_{s}', G^{s})$ for any $s \geq 2$, by \eqref{oleinik property};
\item for any $s > 1$, there exists an open set $U \subset \R^{3}$ and $f \in G^{s}(U) \setminus \Com(U)$ such that $L_{2,1}u \in  \Com(U)$, as a particular consequence of \eqref{oleinik property II}.
\end{itemize}

\end{Exe}

\begin{Exe}
If we write the elements of $\R^{2n+1}$ as $(\tau, t, x)$, where $\tau \in \R$, $t = (t_{1}, \ldots, t_{n})$ and $x = (x_{1}, \ldots, x_{n})$,  consider the following class  of \emph{Fokker-Planck} operators: 
\begin{equation*}
\mathcal{L} = \frac{\del}{\del \tau} + \sum_{j = 1}^{n} t_{j} \frac{\del}{\del x_{j}} - \alpha(t) \Delta_{t},
\end{equation*}
where $a(t)$ is a strictly positive function in $\Com(\R^{n})$.  Note that $\mathcal{L}$ is indeed a \emph{tube operator} as described in the beginning of this subsection. It was proved in \cite{clx} that
\begin{equation*}
\text{$\mathcal{L}$ is $\hypo(\D', G^{s})$, \emph{for any $s \geq 3$}} 
\end{equation*}
Theorem \ref{Gevrey hypo implies ultra hypo} allows us to conclude that
\begin{equation*}
	\text{$\mathcal{L}$ is $\hypo(\D_{s}', G^{s})$, \emph{for any $s \geq 3$}}.
\end{equation*}

\end{Exe} 

\begin{Exe} We now discuss the \emph{Okaji's operators}, defined in Example~\ref{Okaji}, from another point of view. Notice that is another example of \emph{tube operator} described in the beginning of this section. Moreover, it was proved in \cite{oka} that
\begin{equation*}
\text{$P$ is $\hypo (\D', G^{s} )$ if and only if $1 \leq s \leq \displaystyle \frac{4k}{2k-1}$}. 
\end{equation*}
Theorem \ref{Gevrey hypo implies ultra hypo} and Remark \ref{RA case}  allow us to conclude that
\begin{itemize}
	\item 	$P$ is $\hypo(\D_{s}', G^{s})$, \emph{for any} $s \in \Big( 1,  \displaystyle \frac{4k}{2k-1} \Big]$;
	\item  $P$ is $\hypo(\mathcal{B},\Com)$;
	\item let $t  >  \displaystyle \frac{4k}{2k-1}$ and fix $r > t$. There exists an open set $U \subset \R^{2}$ and 
	\begin{equation*}
	\text{$f \in G^{r}(U) \setminus G^{t}(U)$ such that $P f \in  G^{t}(U)$.}
	\end{equation*}
	
\end{itemize}

\end{Exe}

\begin{Exe} 
Consider the \emph{Tricomi operator} on $\R^{2}$, given by
\begin{equation*}
P = \frac{\del^{2}}{\del t^{2}} + t \frac{\del^{2}}{\del x^{2}}, \ \ (x, t) \in \R^{2}.
\end{equation*}
It was proved in \cite{gz} that $P$ is not $\hypo(\Cinf, \Com)$. It follows in particular that for any fixed $s > 1$ there exists $U \subset \R^{2}$ and 
\begin{equation*}
\text{$u \in G^{s}(U) \setminus \Com(U)$ such that $Pu \in \Com(U)$. } 
\end{equation*}
It is worth mentioning that the authors also mention that $P$ can be microlocally conjugated to an operator $D_{t}$, which clearly satisfies the condition above. 
\end{Exe}

\section{Hypoellipticity for the transpose} \label{section6}

Let $\Omega \subset \R^{N}$ be an open set. Following the notation set in \cite{ch1}, we denote by $\mathfrak{S}(\Omega)$ by the class of linear partial differential operators with real-analytic coefficients which satisfy the following condition: 
\begin{equation} \label{L2 solvability}
\begin{split}
\text{For every} \ U & \subset \subset \Omega \ \text{open, there exists $C> 0$ such that} \\
&\left\|\varphi \right\|_{L^{2}} \leq C \left\|\transp{P} \varphi \right\|_{L^{2}}, \ \ \ \forall \varphi \in \Cinf_{c}(U).
\end{split}
\end{equation}

Condition \eqref{L2 solvability} has interesting implications. If $P \in \mathfrak{S}(\Omega)$, we have the following results:

\begin{itemize}[leftmargin=*]
	\item It is a consequence of Hahn-Banach theorem that for every $V \subset \subset\Omega$, there exists a bounded linear operator $Q_{V}: L^{2}(V) \lra L^{2}(V)$ such that $PQ_{V} = I$ in $L^{2}(V)$;
	\item If $P$ is $\hypo (\D', G^{s})$, then $P$ is $\hypo (\D', \Cinf)$ and $\hypo(\D', G^{r})$, for any $r > s$ (\cite[Section III]{met});
	\item If $P$ is $\hypo (\D', \Com)$, its transpose $\transp{P}$ is $\hypo(\mathcal{B}, L^{2}_{\loc})$ (\cite[Theorem 2]{ch1});
	\item If $P$ is $\hypo (\D', G^{s})$, for some $s > 1$, then $\transp{P}$ is $\hypo (\D_{s}', L_{\loc}^{2})$ \cite[Lemma 2.1]{ch2}. 
\end{itemize}

\begin{Obs}
  For the next result we will apply the Closed Graph Theorem of De Wilde as stated in \cite[(2), p. 57]{kot}: \emph{a sequentially closed linear mapping of an ultrabornological space $E$ into a webbed space $F$ is continuous}.

  In our result the \emph{ultrabornological spaces} will be the following:  
  \begin{itemize}
  \item   Since $G_{c}^{r}(U)$ is a DFS space \cite[Theorem 2.6, p. 44]{kom}, it is a \emph{ultrabornological space} \cite[Remark 10, Proposition 12, p. 60]{nac}.
  \item Since $\Cinf_c(U)$ is an LF-space, it is \emph{complete} \cite[pg. 60]{sw} and \emph{bornological} \cite[Cor. 2, p. 62]{sw}, which implies it is \emph{ultrabornological} \cite[p. 44]{kot}.
  \end{itemize}
  On the other hand, the \emph{webbed spaces} will be the following:  
  \begin{itemize}
  \item   $G^{r}(U)$ is a \emph{webbed space} \cite[(4), p. 55 and (7), (8), p. 63]{kot}.
  \item  $\Cinf(U)$ is a Fréchet space and therefore is a \emph{webbed space} \cite[(4), p. 55]{kot}.
  \end{itemize}

\end{Obs}

Using a similar argument to the one applied for the proof of \cite[Lemma 2.1]{ch2}, we obtain the following result.

\begin{Teo}  \label{weak hypo}
Let $P$ be an element of $\mathfrak{S}(\Omega)$ and suppose $P$ is $\hypo (\D', G^{s})$ for some $s > 1$. Then for any $r > s$ we have that $\transp{P}$ is  $\hypo (\D_{s}', \D_{r}')$. Furthermore, $\transp{P}$ is  $\hypo(\D_{s}', \D')$.
\end{Teo}

\begin{proof}
We may assume that $\Omega$ is a relatively compact open subset of $\RN$. Indeed, if we prove that $\transp{P}$ is $\hypo(\D'_s, \D'_r)$ on a family of relatively compact open subsets, the property will also hold true for its union. Furthermore, since $P$ belongs to $\mathfrak{S}(\Omega)$, there exists $Q: L^{2}(\Omega) \lra L^{2}(\Omega)$ such that $PQ f= f$ for every $f\in L^{2}(\Omega)$. 

Fix $U\subset \Omega$ and let $u \in \D_{s}'(U)$ be such that $\transp{P}u  \in \D_{r}'(U)$. We intend to show that $u  \in \D_{r}'(U)$. 
Let  $V \subset \subset U$ be an open set and pick $\chi \in G_{c}^{s}(U)$ such that $\chi \equiv 1$ on a neighborhood of $\overline{V}$; let $W$ be an open set of $U$ containing $\supp \dd \chi$ such that $W\cap \overline{V} = \emptyset$.  If $f\in G^{s}_c(V)$, then we have
\begin{align*} 
\langle u, f \rangle &= \langle \chi u, f \rangle 
= \langle \chi u, PQ f \rangle = \langle {\transp{P}} (\chi u), Qf \rangle,
\end{align*}
where we used that $Q( G^{s}_c(U))\subset G^{s}(U) \cap L^{2}(U)$ since $P$ is $\hypo(\D', G^s)$.

Define $v\in \E'_s(U)$ by
\begin{align*}
 v=  {\transp{P}} (\chi u) - \chi \transp{P} u.
\end{align*}
Given any open set $A \subset U \setminus \supp \dd \chi$ we have
\begin{align*}
\langle v, \psi \rangle &= \langle (\chi u), P(\psi) \rangle - \langle u, P(\chi \psi ) \rangle \\
&=  \langle u, \chi P(\psi) \rangle - \langle u, \chi P( \psi ) \rangle \\
&= 0,
\end{align*}
for all $ \psi \in G_{c}^{s}(A).$
Hence $v \in \E'_s(W)$. By definition,
\begin{align*} 
G^{s}_c(V) \ni f  \mapsto u(f)  = \langle {\transp{P}} u, \chi Qf \rangle+ \langle v, Q f\rangle. 
\end{align*}
We shall prove that the maps
\begin{align*}
  G^{r}_c(V) \ni f \mapsto \langle {\transp{P}} u, \chi Qf \rangle,\quad
  G^{r}_c(V)  \ni f \mapsto \langle v,  Qf \rangle,
\end{align*}
are well defined and continuous. Since $V$ is arbitrary, this implies that  $u$ has an extension as an ultradistribution of order $r$.

The fact that $P$ is $\hypo(\D', G^{s})$ implies that $P$ is also $\hypo(\D', G^{r})$ and so $Q( G^{r}_c(U))\subset G^r(U).$ We claim that the map $Q: G^{r}_c(U) \to G^{r}(U)$ is continuous. Let $(f_n, g_n)$ be a sequence \emph{in the graph of $Q$} such that $f_n \lra f_0$ in $G^{r}_c(U)$ and $g_n \lra g_0$ in $G^{r}(U)$. Note that $Q: L^{2}(U) \lra L^{2}(U)$ is continuous so $\left\{Q f_n \right\}_{n \in \N}$ is a sequence in $L^{2}(U)$ which converges to $ Q f_0 \in L^{2}(U)$. Since $Q f_n =g_n$ converges as a  distribution both to $Qf_0$ and $g_0$ it follows that $Qf_0=g_0$.
It follows from Closed Graph Theorem that $Q$ is continuous and therefore
 \begin{align*}
   G^{r}_c(V) \ni f \mapsto \langle {\transp{P}} u, \chi Qf \rangle ~ \text{is \emph{continuous}}.
 \end{align*}


Next recall that $W \subset U$ is such that $\overline{V} \cap W = \emptyset$. Since $f \in G^{r}_c(V)$ one has $P (Q f)= 0$ on $W$. The fact that $P$ is $\hypo(\D', G^{s})$ implies that $(Q f)|_W\in G^{s}(W)$.
Therefore, we have a linear operator $\lambda: G^{r}_c(V) \lra G^{s}(W)$ given by $\lambda(f) = (Qf)|_W$. Once again we apply the Closed Graph theorem to prove that $\lambda$ is continuous. 

Let $(f_n, g_n)$ a sequence in the graph of $\lambda$ such that $f_n \lra f_0$ in $G^{r}_c(V)$ and $g_n \lra g_0$ in $G^{s}(W)$. Since $Q: L^{2}(\Omega) \lra L^{2}(\Omega)$ is continuous,  $ Q f_n$ is a sequence in $L^{2}(\Omega)$ that converges to $ Q f_0 \in L^{2}(\Omega)$. Note that $(Q f_n)|_W= g_n$ therefore $(Q f_0)|_W = g_0$ in $L^{2}(W)$. It follows that $(Q f_0)|_W = g_0$ and by the closed graph theorem $\lambda: G^{r}_c(V) \lra G^{s}(W)$ is continuous. Since $v \in \E'_s(W)$, the composition $v \circ \lambda: G^{r}_c(V) \lra \C$ is an ultradistribution of order $r$, which proves the first statement.

Next we proceed to the distributional case; let $U \subset \Omega$ be an open set and fix $u \in \D_{s}'(U)$ such that $\transp{P}u  \in \D'(U)$. The machinery is similar: let  $V \subset \subset U$ be an open set and pick $\chi \in G_{c}^{s}(U)$ such that $\chi \equiv 1$ on a neighborhood of $\overline{V}$;  take $W$ an open set of $U$ containing $\supp \dd \chi$ such that $W\cap \overline{V} = \emptyset$. If $f\in G^{s}_c(V)$, then we have
\begin{align*} 
	\langle u, f \rangle &= \langle \chi u, f \rangle 
	= \langle \chi u, PQ f \rangle = \langle {\transp{P}} (\chi u), Qf \rangle .
\end{align*}

Define $v\in \E'_s(W)$ by
\begin{align*}
	v=  {\transp{P}} (\chi u) - \chi \transp{P} u.
\end{align*}
It follows from the definition that
\begin{align*} 
	G^{s}_c(V) \ni f  \mapsto u(f)  = \langle {\transp{P}} u, \chi Qf \rangle+ \langle v, Q f\rangle. 
\end{align*} 
We shall prove that
\begin{align*}
	\Cinf_c(V) \ni f \mapsto \langle {\transp{P}} u, \chi Qf \rangle,\quad
	\Cinf_c(V)  \ni f \mapsto \langle v,  Qf \rangle
\end{align*}
are both well defined and continuous.

 Since $P$ is $\hypo(\D', \Cinf$), we analogously obtain an operator 
$$Q: \Cinf_{c}(U) \to \Cinf(U),$$ 
with the same argument previously applied we deduce that $Q$ has closed graph and the Closed Graph Theorem gives that $Q$ is continuous. Hence
\begin{equation*}
\text{$\Cinf_c(V) \ni f \mapsto \langle {\transp{P}} u, \chi Qf \rangle$ is \emph{continuous}. }
\end{equation*}

%
%

Since $W \subset U$ was taken such that $\overline{V} \cap W = \emptyset$, for any $f \in \Cinf_{c}(V)$ one has $P (Q f)= 0$ on $W$. The fact that $P$ is $\hypo(\D', G^{s})$ implies that $(Q f)|_W\in G^{s}(W)$. Thus we have a linear operator 
\begin{equation*}
\begin{split}
\lambda: \Cinf_{c}(V) &\lra G^{s}(W) \\
f &\mapsto (Qf)|_W.
\end{split}
\end{equation*}
In order to show that 
\begin{equation*}
	\Cinf_c(V)  \ni f \mapsto \langle v,  Qf \rangle \ \text{is \emph{continuous}},
\end{equation*} 
it remains to prove that $\lambda$ is continuous. Once again $\lambda$ has closed graph, which allows us to apply once again  the Closed Graph Theorem, finalizing the proof. 
\end{proof}

\subsection{Hörmander's operators} \label{Hormander Operators}
  
Let $\Omega \subset \RN$ be an open set and $P$ an operator in the following form:
\begin{equation} \label{sums of squares}
P = \sum_{j =1}^{\nu} \XX_{j}^{2} + \XX_{0} + f,
\end{equation}
where the following properties hold true: 
\begin{enumerate}[label=\alph*)]
	\item $\XX_{0}, \XX_{1}, \ldots, \XX_{\nu}$ are real-analytic, real vector fields on $\Omega$; \label{RAVF}
	\item $f \in \Com(\Omega)$; \label{AC}
	\item For each $y \in \Omega$, the \emph{Lie algebra} generated by $\XX_{0}, \XX_{1}, \ldots, \XX_{\nu}$ at $y$ is equal $T_{y}\Omega$. \label{LAC}
\end{enumerate}

\begin{Def}
Let $P$ be an operator as in \eqref{sums of squares}, satisfying properties \ref{RAVF}, \ref{AC} and \ref{LAC}. We shall say that $P$ is a \emph{Hörmander's operator}.
\end{Def}

If $P$ is a Hörmander's operator, the same holds for $\transp{P}$ (see \cite[Remark 14]{bra}). It was famously proved by Hörmander in \cite{hor2} that
\begin{equation} \label{usual hypo}
\text{\emph{both $P$ and $\transp{P}$ are $\hypo(\D', \Cinf)$} on $\Omega$.}
\end{equation}
In fact one has (see for instance \cite{k2})) the  following \emph{subelliptic estimate} for $P$ (and hence for $\transp{P}$ as well): there exists $\epsilon > 0$ such that for  every compact $K \subset \Omega$ there exists $C = C(K) > 0$  for which
\begin{equation} \label{subelliptic estimate}
\left\|u \right\|_{\epsilon} \leq C \left(\left\|\transp{P}u \right\|_{L^2} + \left\|u \right\|_{L^2} \right), \ \ \ \forall u \in \Cinf_c(K).
\end{equation} 

By standard arguments (\cite[Thm. 8.7.2]{hor}), one deduces from \eqref{subelliptic estimate} the following: for every $y \in \Omega$, there exists $U  \subset \subset \Omega$ a relatively compact open neighborhood   of $y$ and $C = C(U) >0$ such that
\begin{equation} \label{l2 solv locally}
\|u \|_{L^2} \leq C \left(\left\|\transp{P}u \right\|_{L^2} \right), \ \ \ \forall u \in \Cinf_c(U).
\end{equation}
That is, Hörmander's operators are elements of $\mathfrak{S}(\Omega)$.

\begin{Teo}\label{Gevrey hypo implies weak hypo}
Let $\Omega \subset \RN$ be an open set and $P$ be a Hörmander's operator on $\Omega$. Given $U  \subset \Omega$, assume that $P$ and $\transp{P}$ are $\hypo(\D', G^s)$ in $U$ for some $s>1$, then the following properties hold:
\begin{itemize}
\item Both \emph{$P$} and $\transp{P}$ are $\hypo (\D_{s}', G^{r})$ on $U$, for any $r \geq s;$
\item Both \emph{$P$} and $\transp{P}$ are $\hypo (\D_{s}', \Cinf )$ on $U;$ 
\item  Both \emph{$P$} and $\transp{P}$ are $\hypo (\D_{s}', \D' )$ on $U;$
\item  Both \emph{$P$} and $\transp{P}$ are $\hypo (\D_{s}', \D_{r}')$ on $U$, for any $r > s$. 
\end{itemize}
\end{Teo}
\begin{proof}
  By combining  that $P$ and $\transp{P}$ are $\hypo(\D', G^s)$ in $U$  with Theorem~\ref{weak hypo}, we deduce that
\begin{equation} \label{weak hypo II}
  \begin{split}
    &\text{both \emph{$P$} and $\transp{P}$ are $\hypo (\D_{s}', \D')$ on $U$,} \\
&\text{both \emph{$P$} and $\transp{P}$ are $\hypo (\D_{s}', \D'_{r})$ on $U$, for any $r> s$. } \\
\end{split}
\end{equation}
Using again that that $P$ and $\transp{P}$ are $\hypo(\D', G^s)$ in $U$, it is then a consequence of \eqref{usual hypo} and \eqref{weak hypo II} that
\begin{equation} \label{hypo distribution gevrey}
\begin{split}
&\text{both \emph{$P$} and $\transp{P}$ are $\hypo (\D_{s}', \Cinf )$ on $U$. } \\
  &\text{both \emph{$P$} and $\transp{P}$ are $\hypo (\D_{s}', G^{r})$ on $U$, for any $r\geq s$. }
\end{split}
\end{equation}

\end{proof}

\begin{Exe}
Now given $y \in \Omega$, we may assume that for the same $U$ in \eqref{l2 solv locally} (otherwise we just shrink the neighborhood) there exists  $m_{y} \in \N$ for which 
\begin{equation} \label{Hormander index}
\begin{split}
&\text{\emph{the vector fields $\XX_{0}, \XX_{1}, \ldots, \XX_{\nu}$ and their commutators }} \\
&\text{\emph{of length up to $m_{y}$ generate $T_xU$ for every $x\in U$.}}
\end{split}
\end{equation}
In that case, we have the following fact (\cite[Thm. $1.4$]{bm}):
\begin{equation} \label{gevrey hypo}
\text{both \emph{$P$} and $\transp{P}$ are $\hypo (\D', G^{s})$ on $U$, for any $s \geq m_{y}$. }
\end{equation}
Therefore, we can always apply Theorem~\ref{Gevrey hypo implies weak hypo} for every $s>1$ satisfying $s\geq m_y$.

\begin{Obs}
	It is worth mentioning  that \eqref{hypo distribution gevrey} was obtained for $s\geq m_y$ in \cite{bm} as well . 
\end{Obs}
\end{Exe}

Notice that the Gevrey regularity we just obtained was a consequence of \eqref{gevrey hypo}, a local result whose first version was proved by Derridj and Zuily \cite{dz}. Generally speaking, the Gevrey regularity provided by the number of commutators is not always the best one can find for an arbitrary Hörmander's operator (see \cite[p.2793]{bm}). 

\begin{Exe}
Recall the \emph{Oleinik  operators} presented in Example \ref{Oleinik operators}. By their expression \eqref{Lpq operator}, it is not difficult to check that $\transp{L}_{p,q} = L_{p,q}$. Previously, we proved  that
\begin{equation*}
	L_{p,q} \ \text{is} \ \hypo (\D_{s}', G^{s}) \ \text{\emph{if and only if}} \ s \geq \frac{p}{q}.
\end{equation*}
Suppose that $p/q > 1$. By Theorem \ref{Gevrey hypo implies weak hypo}, we deduce that
\begin{itemize}
	\item 	$L_{p,q}$ is $\hypo (\D_{p/q}', \D_{r}' )$, for any $r > p/q$;
	\item $L_{p,q}$ is $\hypo (\D_{p/q}', \D' )$;
	\item $L_{p,q}$ is $\hypo (\D_{p/q}', \Cinf )$;
	\item $L_{p,q}$ is $\hypo (\D_{p/q}', G^{s} )$, for any $s \geq p/q$.
\end{itemize}
When $p/q = 1$, one can replace $p/q$ by any $\sigma >1$ on the items  above.

\begin{Obs}
In the particular case of the \emph{Baouendi-Goulaouic operator} $p = 2$ and $q=1$ and the results above hold when the starting space is $\D'_2$. 

On the other hand, the authors prove in  \cite[Corollary $1.1$]{ch2} the following result: for any given $1 < t < 2$, there exists $\Omega \subset \R^{3}$ and
\begin{equation*}
\text{$u \in \D_{t}'(\Omega) \setminus \D'(\Omega)$ such that $L_{2,1}(u) = 0$}, 
\end{equation*}
which proves that the result is sharp in the sense that we cannot replace $\D'_2$ by any slightly bigger space of ultradistributions.
%
%
\end{Obs}
\end{Exe}

\begin{Exe} \label{Bove-Mughetti}
	Consider the following class of operators on $\R^{2}$, studied  in \cite{bm}:
	\begin{equation} \label{bove-mughetti}
		P = \frac{\del^{2}}{\del x_{1}^{2}} + x_{1}^{2(q-1)}\frac{\del^{2}}{\del x_{2}^{2}} - x_{1}^{k}  \frac{\del}{\del x_{2}},
              \end{equation}
            where $k, q$ are integers, $k>0$ and $q\geq 2$. 
	The authors prove \cite[Proposition $4.1$]{bm} that $P$ is $\hypo (\D', \Com)$. It is a consequence of Theorem~\ref{Gevrey hypo implies ultra hypo} and Remark \ref{RA case} that
	\begin{equation*}
	\text{$P$ is $\hypo (\mathcal{B}, \Com)$}.
	\end{equation*}
By performing a simple change of coordinates, one proves that the same holds for $\transp{P}$. It is then a consequence of Theorem \ref{Gevrey hypo implies weak hypo} that
\begin{itemize}
	\item both $P$ and $\transp{P}$ are $\hypo (\D_{s}', G^{s})$, for any $s > 1$,
	\item  both $P$ and $\transp{P}$ are $\hypo (\D_{s}', \D'_{r} )$, for any $r > s > 1$,
	\item  both $P$ and $\transp{P}$ are $\hypo (\D_{s}', \D' )$, for any $s > 1$. 
\end{itemize}
By taking into consideration \cite[Cor. 2, p. 336]{ch1}, we also conclude the following:
\begin{itemize}
\item both $P$ and $\transp{P}$ are $\hypo (\mathcal{B}, \Cinf)$,
\item both $P$ and $\transp{P}$ are $\hypo (\mathcal{B}, G^{s})$, for any $s \geq 1$. 
\end{itemize}

\end{Exe}

\section{Gevrey hypoellipticity for a class of systems of vector fields} \label{LIS Gevrey}

In this section, we intend to prove results which are similar to the ones proved in Section~\ref{constant coefficients operators} in the framework of \emph{Gevrey Locally Integrable Structures} (for more details on the subject, we recommend \cite{bch} or \cite{tre3}).

Consider $\mathcal{V}$  a Gevrey involutive  structure of order $s_{0}>1$ and rank $n\in \Z_+$ defined on a neighborhood of the origin $\Omega \subset \R^{n+m}$, that is, each point $x_{0}\in \Omega$ has an open neighborhood $W_{x_{0}}\subset \Omega$ on which one can find $\LL_{1}, \ldots, \LL_{n}$  vector fields of class $G^{s_0}$ defining a frame for $\mathcal{V}|_{W_{x_{0}}}$. The orthogonal of $\mathcal{V}$ with respect to the duality of vector fields and forms is the subbundle $\T$ of $\C T^{\ast} \Omega$ defined by
\begin{align*}
	\T_{x_{0}}= \{ \lambda \in \C T^{\ast}_{x_{0}} \Omega: \lambda(v)=0, \; \forall v \in \mathcal{V}_{x_{0}}\}.
\end{align*}

\begin{Def}
$\mathcal{V}$ is a \emph{Gevrey locally integrable structure of order $s_0>1$} if $\T$ is locally generated by the differential of $m$ linearly independent functions of class $G^{s_0}$.
\end{Def}

It can be proved (see \cite{bch}) that if $(U, x_1, \ldots, x_m, t_1, \ldots, t_n)$ is an open neighborhood of the origin sufficiently small, there exist $Z_1, \ldots, Z_m: U \lra \C$ such that
\begin{align*}
	Z_j= x_j + i \phi_j(t,x),
\end{align*}
where each $\phi_j:U\lra \R$ is a function of class $G^{s_0}$ and
\begin{align*}
	\T|_{U}= \Span \{ \dd Z_1, \ldots, \dd Z_m\}\big|_U.
\end{align*}
In this situation, one can find $\LL_1, \ldots, \LL_n$ generating $\mathcal{V}$ over $U$ and complete to a basis of $\C T \Omega$ over $U$ with vector fields $\MM_1, \ldots, \MM_m$ for which the de Rham operator can be written as
\begin{align*}
	\dd f=  \sum_{j=1}^{n} \left( \LL_j f \right) \dd t_j+ \sum_{k=1}^{m} \left(\MM_k f \right) \dd Z_k, \ \ \ \text{for every $f\in \Cinf(U).$}
\end{align*}

We can now give a concrete representation for the quotient bundle $\C T^{\ast} \Omega/\T$ over $U$ by taking representatives $\dd t_1, \ldots, \dd t_n$. Thus the operator induced by $\dd$ on ultradistributional sections of  $\LLambda \dfn\C T^{\ast} \Omega/\T$ can be represented by $\mathbb{L}: \D_{s_{0}}'(U) \to \D_{s_{0}}'(U; \LLambda^{1})$, where 
\begin{equation*}
	\D_{s_{0}}'(U; \LLambda^{1}) = \bigg\{ \sum_{j=1}^{n} u_{j} \dd t_{j}, \ \ u_{j} \in \D_{s_{0}}'(U) \bigg\}
\end{equation*}
is the space of $1$-ultracurrents generated only by $\dd t_1, \ldots, \dd t_n$. Here the action of $\mathbb{L}$ in $v \in \D_{s_0}'(U)$ is given by

\begin{equation} \label{action of L}
	\mathbb{L}v = \sum_{j=1}^{n} \left(\LL_{j}v \right)\dd t_{j}.
\end{equation}
Note that one may define the concept of hypoellipticity for pairs with respect $\mathbb{L}$ in the same manner we did for operators previously.  

Next we recall certain results proved in \cite{rag}. Given $s > s_{0}$ and assuming that $U$ is a neighborhood of the origin sufficiently small,  there exists a local ultradifferential operator $Q_{s}$ (in order to check its expression, see \cite[Section 9.1]{rag}) satisfying the following conditions:
\begin{enumerate} 
	\item 		 \label{continuity in Gs0}
	The operator $Q_{s}$ acts continuously on $G^{s_{0}}(U)$ and $\D_{s_{0}}'(U)$, (see \cite[Lemma 5.5]{rag});

	\item  \label{commutativity}
	$Q_{s}$ commutes with $\LL_{j}$ in $\D_{s_{0}}'(U)$ for every $ j \in \{1, \ldots, n \}$;
	\item  Each $x_{0} \in U$ has an open neighborhood $U_{x_{0}} \subset \subset U$ such that
	\begin{equation} \label{representation theorem}
		\forall u \in \E_{s}'(U_{x_{0}}), \ \exists f \in G^{s}(U_{x_{0}}); \ Q_{s}f = u \ \text{in $\D_{s_{0}}'(U_{x_{0}})$}
              \end{equation}
              (see \cite[Theorem 6.2]{rag});
	\item Fix $s' > 1$ such that $s_{0} < s' < s$.  Then for each $V \subset U_{x_{0}}$  open neighborhood of $x_0$,  there exists $W \subset \subset V$ an open neighborhood of $x_0$ for which the following property holds true: 
	\begin{equation} \label{regularity kernel}
		u \in \D_{s'}'(V), \ \ Q_{s} u = 0 \ \ \Longrightarrow \ \ u \big{|}_{W} \in G^{s_{0}}(W) 
              \end{equation}
           (see \cite[Proposition 7.2]{rag}).  
              	\item $Q_{s}$ has the following property (see Remark~\ref{remarkGrVaiemGr}): 
	\begin{equation} \label{representation theorem Gs0}
		\text{ if $g \in G^{s_{0}}(U)$, $v \in \D_{s_{0}}'(U)$ and $Q_{s}v = g$, then $v \in  G^{s_{0}}(U)$}
              \end{equation}

\end{enumerate}
It is also worth mentioning that $Q_{s}$ also acts on spaces of ultracurrents, coefficientwise; furthermore, it commutes with $\mathbb{L}$.

\begin{Def}
We say that $u\in \D'_{s_0}(U; \LLambda^{1})$ is \emph{$\mathbb{L}$-closed} if $\mathbb{L}u=0$.
\end{Def}  

The following generalizations are obtained for $\mathbb{L}$-closed forms (see \cite[Section 9]{rag}):
\begin{enumerate}[label=(\alph*)]
	\item \label{property 40}  Each $x_{0} \in U$ has an open neighborhood $U_{x_{0}} \subset \subset U$ such that if $u \in \D'_s(U; \LLambda^{1})$ is $\mathbb{L}$-closed and can be written as $Q_{s}g$, with $g \in G^{s}(U; \LLambda^{1})$, there exists an $\mathbb{L}$-closed one-form $f \in G^{s}(U_{x_0}; \LLambda^{1})$ such that $u = Q_{s}f$ on $U_{x_{0}}$;
	\item \label{property 41}  If $g \in  G^{s_{0}}(U; \LLambda^{1})$ is $\mathbb{L}$-closed and $v \in \D'_s(U; \LLambda^{1})$ is such that $Q_{s}v = g$, then $v \in G^{s_{0}}(U; \LLambda^{1})$;
	\item \label{property 42}   If $u \in \D'_s(U; \LLambda^{1})$ is $\mathbb{L}$-closed and $Q_s u=0$, for any fixed $x_{0} \in U$ there exist an open neighborhood  $W_{x_{0}} \subset U$ and  $v \in G^{s_0}(W_{x_{0}})$ such that $u=\mathbb{L} v$ on $W_{x_{0}}$. 
\end{enumerate}


\begin{Def}
$\mathbb{L}$ is said to be \emph{$(\Cinf, \Cinf)$-locally solvable} on $U$ when the following property holds: for any given $x_0 \in U$ and for any $U_{x_0}\subset U$ open neighborhood of $x_0$, we can find another open neighborhood $V_{x_0}\subset U_{x_0}$  such that
\begin{align*}
&\text{every $ u \in \Cinf(U_{x_0}, \LLambda^{1})$ which is $\mathbb{L}$-closed is $\mathbb{L}$-exact on $V_{x_0}$}. 
\end{align*}
That is, there exists $v \in \Cinf(V_{x_0})$ such that $ \mathbb{L}v = u $ on $V_{x_0}$.
One can define the notion of $\mathbb{L}$ being {\it $(\D', \D')$-local solvable} in degree $1$ in a similar way.
\end{Def}

We are prepared to state and prove the main result of this section.

\begin{Teo} \label{hypo systems}
	Let $\mathcal{V}$ be a $G^{s_0}$ locally integrable structure and  $\mathbb{L}$ the differential operator described in \eqref{action of L}. Then for any $s>s_0$, the following properties are equivalent:
	\begin{enumerate}
		\item $\mathbb{L}$ is $\hypo(\D'_{s}, G^{s_{0}})$; \label{eqs11}
		\item $\mathbb{L}$ is $\hypo(\D', G^{s_{0}})$; \label{eqs22}
		\item $\mathbb{L}$ is $\hypo( \Cinf, G^{s_{0}})$; \label{eqs33}
		\item $\mathbb{L}$ is $\hypo( G^{s}, G^{s_{0}})$. \label{eqs44}
	\end{enumerate}
	Furthermore, if $\mathbb{L}$ it is $(\Cinf,\Cinf)$-locally solvable in degree $1$, any of the conditions above imply the following properties:
	\begin{enumerate}\setcounter{enumi}{4}
		\item $\mathbb{L}$ is $\hypo( \D'_{s}, \D')$; \label{eqs55}
		\item $\mathbb{L}$ is $\hypo( \D'_{s}, \Cinf)$. \label{eqs66}
	\end{enumerate}
\end{Teo}
\begin{proof}
  For the first part of the theorem, we only need to prove that \eqref{eqs44} implies \eqref{eqs11}. 
  So let us assume  that \eqref{eqs11} fails.   
  Therefore, there exists an open set $U_0 \subset U$,
	\begin{equation*}
		\text{ $u \in \D'_{s}(U_0) \setminus G^{s_{0}}(U_0)$  and $f \in G^{s_{0}}(U_0; \LLambda^{1})$ such that $\mathbb{L}u = f$.}
	\end{equation*}
Therefore, there is $x_0\in U_0$ such that the germ of $u$ at $x_0$ is not a $G^{s_0}$ function on any neighborhood of $x_0$. 
        
Let $V \subset \subset  U_0$ be an open neighborhood of $x_0$. Pick $\chi \in G_{c}^{s}(U_0)$ such that $\chi \equiv 1$ on $V$ and define $v = \chi u \in \E_{s}'(U_0)$. As a consequence of \eqref{representation theorem}, there exists  $\varphi \in  G^{s}(U_0)$ such that $Q_{s} \varphi = v$. Since $Q_{s}$ is local and acts continuously on $G^{s_{0}}$ and $v\in \mathcal{D}'_{s}(V) \setminus G^{s_0}(V)$ it follows that $\varphi$ is not of class $G^{s_{0}}$ on any neighborhood of $x_0$.
Moreover,
\begin{align*}
  f \big{|}_{V} = \mathbb{L} u\big{|}_{V}= \mathbb{L} v\big{|}_{V} = \mathbb{L} (Q_{s} \varphi)\big{|}_{V} = Q_{s} \left(\mathbb{L} \varphi \big{|}_{V} \right).
\end{align*}

It follows from Property~\ref{property 40} the existence of $W \subset \subset V$ an open neighborhood of $x_0$ and an $\mathbb{L}$-closed form $g \in G^{s}(W; \LLambda^{1})$ such that
\begin{equation*}
 f = Q_{s}g \ \text{on $W$}  
\end{equation*}	
Statement~\ref{property 41} implies that $g \in G^{s_{0}}(W; \LLambda^{1})$.
Hence $ g - \mathbb{L} \varphi$ is an $\mathbb{L}$-closed form in  $G^{s}(W; \LLambda^{1})$ such that $Q_{s} ( g - \mathbb{L} \varphi) = 0$. Finally, by Property~\ref{property 42} there exists $W'\subset W$ another open neighborhood of $x_0$  and $h \in G^{s_{0}}(W')$ for which
\begin{equation*}
 \ \ \mathbb{L}h = g - \mathbb{L} \varphi \ \ \text{on $W'$} \ \ \Rightarrow \  \mathbb{L} \varphi = g - \mathbb{L}h \in   G^{s_{0}}(W'; \LLambda^{1}),
\end{equation*}
and since $\varphi$ is not of class $G^{s_0}$ near $x_0$ \eqref{eqs44} fails as well.

	The next step is to show that \eqref{eqs11} implies  \eqref{eqs55} and \eqref{eqs66}.  First we recall that it was proved in \cite{rag} that local solvability in the $(\Cinf, \Cinf)$ sense is equivalent to local solvability in the $(\D^{'}, \D^{'})$ sense. Let us prove that condition~\eqref{eqs55} holds; suppose that $ V \subset U$ is an open set, $u \in \D'_{s}(V)$ and $f \in \D'(V; \LLambda^{1})$ are such that $\mathbb{L} u = f$. Given $x_{0} \in V$, the $(\D', \D')$-local solvability  of $\mathbb{L}$ implies that  we can find $U_{x_{0}} \subset V$ an open neighborhood of $x_0$  and $\omega \in \D'(U_{x_{0}})$ such that $\mathbb{L}\omega = f$ in $U_{x_{0}}$. Thus
	\begin{equation*}
		\mathbb{L}(u - \omega) = 0 \ \Longrightarrow \ u - \omega \in G^{s_{0}}(U_{x_{0}}) \ \Longrightarrow \ u \in  \D'(U_{x_{0}}).
	\end{equation*}
	 
	Since this argument can be applied for each $x_{0} \in V$,  $u \in \D'(V)$, proving that \eqref{eqs55} holds true. Analogously we verify \eqref{eqs66}, which finalizes the proof. 
\end{proof}

\begin{Exe}
Consider once again the family of Mizohata vector fields on the plane,  described by
\begin{equation*}
	\LL = \frac{\del }{\del x} + ix^{k} \frac{\del }{\del y}, 
      \end{equation*}
 for any integer $k\geq 1$. 
Let $\mathcal{V}$ be the real-analytic locally integrable structure generated by $\LL $, whose first integral is given by
\begin{equation*}
Z_{k}(x,y) = \frac{x^{k+1}}{k+1} + iy, \ \ \forall (x,y) \in \R^{2}.  
\end{equation*} It was mentioned in Example \ref{Mizohata} that
\begin{equation*}
\text{\emph{when $k$ is even, $\LL $ is $\hypo(\D', G^{s})$, for  any $s \geq 1$}.}
\end{equation*}
As a consequence of Theorem \ref{hypo systems}, one has the following properties:
\begin{itemize}
	\item For any $s > 1$ and $r > s >1$,  $\LL $ is $\hypo(\D_{r}', G^{s})$;
	\item For any $s >1$, $\LL $ is $\hypo(\D_{s}', \Com)$ (since any real-analytic function is $G^{s}$ and $\LL $ is $\hypo(\D', G^{s})$).
\end{itemize}
It is worth noting that the results which follow from Theorem \ref{Gevrey hypo implies ultra hypo} and Remark \ref{RA case} are stronger: in fact, one has
\begin{itemize}
	\item For any $s > 1$,  $\LL $ is $\hypo(\D_{s}', G^{s})$;
	\item  $\LL $ is $\hypo(\mathcal{B}, C^{\omega})$ .
\end{itemize}

Nevertheless it is a consequence of the \emph{Nirenberg-Treves condition} (see \cite{nt1}, \cite{nt2}) that $\LL $ is \emph{$(\Cinf, \Cinf)$-locally solvable} on $\R^{2}$, which implies that
\begin{itemize}
	\item For any $s > 1$,  $\LL $ is $\hypo(\D_{s}', \D')$;
	\item For any $s >1$,  $\LL $ is $\hypo(\D_{s}', \Cinf)$.
\end{itemize}

On the other hand, it follows from \cite[Theorem  $2.3.6$]{rod} that 
\begin{equation*}
\text{\emph{when $k$ is odd, $\LL $ is \textbf{not} $\hypo(\D', G^{s})$, for  any $s \geq 1$}.}
\end{equation*}
Then Theorem \ref{hypo systems} implies that
\begin{itemize}
	\item Given $s > 1$ and $r > s$, there exist $\Omega \subset \R^{2}$ and $g \in G^{r}(\Omega) \setminus G^{s}(\Omega)$  such that $\LL g \in G^{s}(\Omega)$.
\end{itemize}
Note that the result may obtained as well from Theorem \ref{Gevrey hypo implies ultra hypo}.  Furthermore one can also conclude from Theorem \ref{Gevrey hypo implies ultra hypo} that
\begin{itemize}
	\item For any $s > 1$, there exist $\Omega \subset \R^{2}$ and $g \in G^{s}(\Omega) \setminus \Com(\Omega)$  such that $\LL g \in \Com(\Omega)$.
\end{itemize}
\end{Exe}

\begin{Exe} \label{key}
We consider now the following generalization of the Mizohata structure: real-analytic locally integrable structures of \emph{tube type} and \emph{corank one} (for more details on the subject we recommend \cite{tre3}). In this context: 
\begin{itemize}
	\item $\Theta$ is an open ball centered at the origin in $\R^{n}$;
	\item  $\phi: U \to \R $ is a real-analytic map such that $\phi(0) = 0$, where $U$ is an open neighboorhood of $\overline{\Theta}$;
	\item The bundle $\mathcal{V}$ is defined on $\R \times \Theta \ni (y, x)$  and is generated by the vector fields
	\begin{equation*}
	\LL_{j} = \frac{\del}{\del x_{j}} - i  \frac{\del \phi}{\del x_{j}} (x) \frac{\del}{\del y}, \ \ \ j = 1, \ldots, n. 
	\end{equation*}
\end{itemize}

Now suppose that the following condition:
\begin{equation*} 
\text{$\phi: \Theta \to \R$ is an \emph{open map}}.
\end{equation*}
In that case, it was proved in \cite{mai}, \cite{batr} and \cite{ccp} that the operator $\mathbb{L}$ (see \eqref{action of L}) is $\hypo (\D', \Cinf)$, $\hypo (\D', \Com)$ and  $\hypo (\D', G^{s})$ for any  $s >1$, respectively. It follows from Theorem~\ref{hypo systems} that
\begin{itemize}
	\item For any $s > 1$ and $r > s >1$,  $\mathbb{L}$ is $\hypo(\D_{r}', G^{s})$;
	\item For any $s >1$, $\mathbb{L}$ is $\hypo(\D_{s}', \Com)$.
\end{itemize}

\

Consider next the particular case given by
\begin{equation} \label{generalization mizohata}
	\LL_{j} = \frac{\del}{\del x_{j}} + i x_{j}^{k_{j}} \frac{\del}{\del y}, \ \ \ k_{j} \in \Z, \ \ j = 1, \ldots, n. 
\end{equation}
Then $\phi(x) = - \displaystyle \sum_{j = 1}^{n} \frac{x_{j}^{k_{j}+ 1}}{k_{j}+ 1}$. Now $\phi$ is an open map if and only if it has no local minimum or maximum. Hence we deduce the following: suppose that there exists at least one $k_{j}$ in \eqref{generalization mizohata} which is even, then $\mathbb{L}$ is $\hypo(\D', G^{s})$ for any $s\geq 1$.
Therefore  once again one concludes from Theorem~\ref{hypo systems} that
\begin{itemize}
	\item For any $s > 1$ and $r > s >1$,  $\mathbb{L}$ is $\hypo(\D_{r}', G^{s})$.
	\item For any $s >1$, $\mathbb{L}$ is $\hypo(\D_{s}', \Com)$.
\end{itemize}

\end{Exe}

\begin{Exe} Consider the following locally integrable structure on $\R^{4}$, studied by H. M. Maire in \cite{mai}: the vector fields which generate the bundle $\mathcal{V}$ are given by
\begin{equation*}
\begin{split}
\LL_{1} &= \frac{\del }{\del x_{1}} + i \left(3 \frac{\del }{\del y_{1}}- (4x_{1}x_{2} + 3)x_{1}^{2}\frac{\del }{\del y_{2}}\right), \\
\LL_{2} &= \frac{\del }{\del x_{2}} - ix_{1}^{4} \frac{\del }{\del y_{2}},
\end{split}
\end{equation*}
for every $ (x_{1}, x_{2}, y_{1}, y_{2}) \in \R^{4}$ and the first integrals are 
\begin{equation*} 
	\begin{split}
	Z_{1}(x, y)&= y_{1} - 3ix_{1}, \\
	Z_{2}(x,y) &= y_{2} + i(x_{1}x_{2} + 1)x_{1}^{3}.
	\end{split}
\end{equation*}
The authors prove in \cite{ccp} that 
\begin{equation*}
\text{\emph{ $\mathbb{L}$ is \textbf{not} $\hypo(\D', G^{s})$, for  any $s \geq 4$}.}
\end{equation*}
In fact they prove the existence of a neighborhood of the origin $\Omega \subset \R^{4}$ and 
\begin{equation*}
u \in \mathcal{C}(\Omega) \setminus \mathcal{C}^{1}(\Omega); \ \ \mathbb{L}u \in G^{4}(\Omega; \LLambda^{1}).
\end{equation*} 
Theorem \ref{hypo systems} implies that given $s \geq 4$ and $r > s$,
 there exist $U \subset \R^{4}$ an open set and $v \in G^{r}(U) \setminus G^{s}(U)$ such that $ \mathbb{L}v \in G^{s}(U; \LLambda^{1})$.

\end{Exe}

\section{Local solvability and hypoellipticity}\label{section8}

In Sections \ref{constant coefficients operators}, \ref{section4} and \ref{LIS Gevrey} we used the concept of local solvability to verify hypoellipticity. First we describe the concepts involved in an abstract manner. 

%

Let $\mathcal{G}\subset \mathcal{F}$ be a pair of sheaves over $\Omega$ and consider $P$ an endomorphism of $\mathcal{F}$ can be restricted to $\mathcal{G}$ as an endomorphism. 
  We recall that $\ker P\subset \mathcal{G}$ if for every $U$ open subset of $ \Omega$ the following property is satisfied: $\{ u \in \mathcal{F}(U): Pu=0\} \subset \mathcal{G}(U)$.

\begin{Def}
  We shall say that $P$ is  \emph{locally solvable with respect to} $\mathcal{G}$ if for any fixed open set $U\subset\Omega$ and $u\in \mathcal{G}(U)$, every $p\in U$ has an open neighborhood $V\subset U$ where one can find $f\in \mathcal{G}(V)$ such that $Pf=u|_V$.
\end{Def}

In the particular case of $\mathcal{G}= \Cinf$, we will also say that $P$ is $\Cinf$-locally solvable. 

\begin{Pro}
	Assume that $P$ is locally solvable with respect to $\mathcal{G}$ and  $\ker P\subset \mathcal{G}$, then $P$ is $\hypo(\mathcal{F},\mathcal{G})$.
\end{Pro}
\begin{proof}
	Let $u\in \mathcal{F}(U)$ such that $Pu \in \mathcal{G}(U)$. The local solvability with respect to $\mathcal{G}$ yields that every $p\in U$ has an open neighborhood $V$ and $f\in \mathcal{G}(V)$ such that $Pf=(Pu)|_V= Pu|_V$. Since $u|_V-f\in \ker P|_V$ it follows that $u|_{V}-f\in \mathcal{G}(V)$ and consequently $u|_V\in \mathcal{G}(V)$. Now this holds true for a neighborhood of each point of $U$, which implies that $u \in \mathcal{G}(U)$, as we intended to prove.
\end{proof}

Now recall that in Example \ref{Lewy} we used the fact that the Lewy operator is not locally solvable in order to deduce that it is not hypoelliptic in the classic sense. Next we prove a similar result in a different context, i.e., a certain type of hypoellipticity can imply local solvability for the transpose. 

\subsection{Implications on solvability}

\begin{Def}
Let $r \in \R$ and $K \subset \R^{n}$ a regular compact set, that is, $K$ is the closure of some bounded  open set.  We define
\begin{equation*}
H^{r}(K) \dfn \left\{u \in H^{r}(\R^{n}):  \supp \ u \subset K \right\}.
\end{equation*}
\end{Def}
$H^{r}(K)$ is a closed subspace of $H^{r}(\R^{n})$ and hence a Hilbert space. Furthermore, it follows from Rellich Lemma that 
\begin{equation*}
r_{1} < r_{2} \ \Longrightarrow  \text{ the inclusion } \iota: H^{r_{2}}(K) \hookrightarrow H^{r_{1}}(K) \ \text{is compact.}
\end{equation*}
Thus if we define, with the inductive limit topology, the space of distributions supported on $K$ by
\begin{equation*}
\E'(K) \dfn \lim_{\stackrel{\longrightarrow}{n \in \N}} H^{-n}(K), 
\end{equation*}
it follows that $\E'(K)$ is a DFS space (for more details, we recommend \cite{kom} or \cite[Appendix A]{mor}). Moreover, it is not difficult to check that $\E'(K)$ can be identified as the dual space of the quotient $\Cinf(K) \dfn \Cinf(\R^{n}) / \mathscr{I}(K)$, where
\begin{equation*}
\mathscr{I}(K) \dfn \left\{f \in \Cinf(\R^{n}):  f  \text{ vanishes to infinite order at each point of $K$} \right\}.
\end{equation*}
Note that if $P$ is a  linear partial differential operator with smooth coefficients, defined on an open set $\Omega \subset \R^{n}$, then for any regular $K \subset \Omega$ it induces a map $P: \Cinf(K) \to \Cinf(K)$.

On the other hand, fix $s > 1$. For any $h > 0$, let $G^{s, h}(K)$ be the space of all elements of $f \in \Cinf(K)$ which satisfy the following estimates:
\begin{equation*}
 \sup_{x \in K} |D^{\alpha} f(x)| \leq C h^{|\alpha|} |\alpha|!^{s}, \ \ \forall \alpha \in \Z_{+}^{n},
\end{equation*}
for some constant $C > 0$. In this context we endow $G^{s, h}(K)$ with the following norm:
\begin{equation*}
\left\|f \right\|_{G^{s, h}(K)} \dfn \sup_{\alpha \in \Z_{+}^{n} } \left(\sup_{x \in K}  \frac{|D^{\alpha} f(x)|}{h^{\alpha} |\alpha|!^{s}} \right).
\end{equation*}
This turns $G^{s, h}(K)$ into a Banach space. 
\begin{Def}
We define $G^{(s)}(K)$ as the space given  by  the following projective limit:
\begin{equation*}
G^{(s)}(K) \dfn \lim_{\stackrel{\longleftarrow}{n \in \N}} G^{s, 1/n}(K) 
\end{equation*}
\end{Def}
\begin{Obs}
Note that $G^{(s)}(K)$ is a Gevrey space of \emph{Beurling type}. 
\end{Obs}
It is possible to prove that $G^{(s)}(K)$ (\cite[Theorem 2.6]{kom}) is an FS space. Let  $\E_{(s)}'(K)$ be its dual space; then it is a DFS space, given by the inductive limit
\begin{equation} \label{Beurling distributions on a compact}
\E_{(s)}'(K) \dfn \lim_{\stackrel{\longrightarrow}{n \in \N}} G^{s, 1/n}(K)', 
\end{equation}
with the inclusions $G^{s, 1/n_{0}}(K)'\hookrightarrow G^{s, 1/n_{1}}(K)'$ being compact whenever $n_{1} > n_{0}$. Furthermore, it is not difficult to see  the inclusions
\begin{equation*}
G^{s, h}(K)  \hookrightarrow \Cinf(K) \ \text{are continuous with dense image, for every $h > 0$.}
\end{equation*}
This implies in particular, by transposition, that
\begin{equation}\label{inclusion first space}
\text{$\E'(K)$ is a subspace of $(G^{s, 1/n}(K))'$, for any $n \in \N$. }
\end{equation}

Let  $\left\{K_{m}\right\}_{m \in \N}$ be a sequence of regular compact sets such that $K_{j} \subsetneq K_{\ell}$ for any $j < \ell$ and $\Omega = \displaystyle \bigcup_{k = 1}^{\infty} K_{j}$.  The \emph{Gevrey-Beurling space of order $s$} is then defined by the projective limit
\begin{equation*}
G^{(s)}(\Omega)   \dfn \lim_{\stackrel{\longleftarrow}{n \in \N}} G^{(s)}(K_{n}).
\end{equation*}
This fact allows us to conclude that its dual is given by
\begin{equation*} 
  \E_{(s)}'(\Omega) \dfn \lim_{\stackrel{\longrightarrow}{n \in \N}} \E_{(s)}'(K_{n}). 
\end{equation*}
Hence, for any regular compact set $K \subset \Omega$, we conclude that
\begin{equation} \label{compact supoort is a subspace}
\text{$\E_{(s)}'(K)$ is a subspace of $\E_{(s)}'(\Omega)$.} 
\end{equation}
%
%
It is not difficult to show the following assertion:
\begin{Pro} \label{relation between local and semiglobal}
Let  $x_0 \in \Omega$ and let $K$ be a regular compact. If $K$ is a neighborhood of $x_0$ such that $P:\Cinf(K) \to \Cinf(K)$ is surjective, then $P$ is $\Cinf$-locally solvable at $x_0$.
\end{Pro}
From now on we are going to work under the assumption that $P$ is $\hypo(\D'_{(s)}, \D')$; in this case we are implicitly assuming that $P$ has coefficients of class $G^{(s)}$. 

\begin{Pro} \label{closed image and finite kernel}
Assume that $P$ is $\hypo(\D'_{(s)},\D')$. For any regular compact set $K \subset \Omega$, consider the induced map given by $P: \E'(K) \to \E'(K)$. Then its range is closed and its kernel is finitely generated. 
\end{Pro}
\begin{proof}
Denote by $\mathrm{Gr(P)}$ the graph of $P$. We claim that $\mathrm{Gr(P)}$ is closed as a subspace of $\E_{(s)}'(K) \times \E'(K)$.  Indeed, let $\left\{u_{n}\right\}_{n \in \N} \subset \E_{(s)}'(K)$ be a sequence such that
\begin{equation*}
\text{$u_{n} \to u$ in  $\E_{(s)}'(K)$ and $Pu_{n} \to v \in  \E'(K)$.}
\end{equation*}
Then $Pu = v$ in $\E_{(s)}'(K)$; since $P$ is $\hypo(\D'_{(s)}, \D')$, it follows that $u \in \E'(K)$ and $Pu = v$ in $\E'(K)$.  It follows from the fact that both $\E'(K)$ and $\E_{(s)}'(K)$ are DFS spaces, together with validity of the inclusion \eqref{inclusion first space}, that we can apply \cite[Theorem 2.5]{ara} to conclude that $P$ has closed range and its kernel is finitely generated.
\end{proof}

Next we state a general fact about linear partial differential operators  with finitely generated kernel. In order to do it we introduce the following notion:

\begin{Def}
Given $x_0\in \Omega$, we shall say that $P$ has a \emph{solution of delta type at $x_0$} if  there exists $u \in \E'(\Omega)$ such that $\supp u= \{x_0\}$ and $Pu =0$. 
\end{Def} 

\begin{Pro} \label{compact injectivity} Fix $x_0 \in \Omega$ and assume that $P$ has no solution of delta type at $x_0$. If the kernel of $P: \E'(\Omega) \lra \E'(\Omega)$ is finitely generated, there exists a sufficiently small $r > 0$ such that   $P: \E'(\overline{B_r(x_0)}) \to \E'(\overline{B_r(x_0)})$ is injective.
\end{Pro}
\begin{proof}
Let  $r$ be any positive real number   such that $\overline{B_{r}(x_0)} \subset \Omega$ and denote by $K_r$ the kernel of the operator
\begin{align*}
P:  \E'(\overline{B_{r}(x_0)}) \lra \E'(	\overline{B_{r}(x_0)}),
\end{align*}
which is finite-dimensional by hypothesis. Now fix $r_0>0$ and let $p =  \dim K_{r_0}$; consider $\{u_1, \ldots, u_p\}$  a basis of $K_{r_0}$. Since  $\supp u_{n} \neq \left\{x_0 \right\}$ for each $n \in \left\{1, \ldots, p \right\}$, there exists $0<r_1<r_0$ such that
	\begin{align*}
		\bigcup_{j=1}^{p} \supp  u_j \setminus \overline{B_{r_1}(x_0)} \neq \emptyset.
	\end{align*}
This implies that  $\dim K_{r_1} \leq p-1$. Therefore, by proceeding recursively we are able to find $r>0$ such that $K_r=\{0\}$, which is equivalent to say that $P$ is injective, finalizing the proof. 
\end{proof}

Propositions~\ref{closed image and finite kernel} and \ref{compact injectivity} yield the following:

\begin{Cor}\label{closed and injective}
  Suppose that $P$ is $\hypo(\D'_{(s)}, \D')$ and has no solution of delta type at $y\in \Omega$. Then there exists a regular compact neighborhood $K$ of $x_0$ for which $P: \E'(K) \to \E'(K)$ is injective and has closed range.
\end{Cor}

\begin{Teo} \label{solution delta type}
Suppose that $P$ has no solution of delta type at $x_0\in \Omega$. If $P$ is $\hypo(\D'_{(s)}, \D')$ for some $s > 1$, its transpose $\transp{P}$ is $\Cinf$-locally solvable at $x_0$. 
\end{Teo}
\begin{proof}
It follows from Corollary~\ref{closed and injective} the existence of  $K$ a compact neighborhood of $x_0$ such that $P: \E'(K) \to \E'(K)$ is injective and has closed range. On the one hand the injectivity implies that the image of $\transp{P}: \Cinf(K) \to \Cinf(K)$ is \emph{dense}. On the other hand, the fact that $P$ has closed range implies that the same holds for its transpose (\cite[Lemma 2.2]{ara}). Thus the operator $\transp{P}: \Cinf(K)\lra \Cinf(K)$ is surjective. The result is then consequence of Proposition \ref{relation between local and semiglobal}.
\end{proof}

One may ask when a differential operator $P$ has no solutions of delta type. We recall now how one can use Holmgren theorem to find properties on $P$ which guarantee the property in the sense of hyperfunctions. Consider $P$ a differential operator on  an open subset of $\Omega \subset\R^n$, of \emph{order $m$} and with \emph{real-analytic coefficients}. 

\begin{Def}
We shall say that $P$ \emph{does not degenerate at $x_{0} \in \Omega$} if its symbol has the following property:
\begin{equation} \label{no singularities}
	\text{$P =  \sum_{|\alpha| \leq m} a_{\alpha}(x) \del_{x}^{\alpha}$} \ \Longrightarrow \  \displaystyle \sum_{|\alpha| = m} |a_{\alpha}(x_0)| > 0. 
\end{equation} 
That is, there exists at least one coefficient of the \emph{principal symbol} of $P$ which does not vanish at $x_{0}$. 
\end{Def}

Next we recall a result which can be found for instance in \cite[Theorem 1.2]{bon}.
\begin{Teo}[Holmgren] \label{Holmgren Theorem}
Let $x_{0} \in \Omega$ and $\Phi: \Omega \to \R$ a $\mathcal{C}^{2}$ function such that
\begin{itemize}
	\item $\Phi(x_{0}) = 0,$
	\item $\dd\Phi (x) \neq 0, \ \forall x \in \Omega$, and  $\dd\Phi (x_{0}) = \xi_{0},$
	\item $\Omega' \dfn \left\{x \in \Omega: \Phi(x) > 0\right\}$.  
\end{itemize} 
Consider $u \in \mathcal{B}(\Omega)$ such that $Pu = 0$. If $P_{m}(x_{0}, \xi_{0}) \neq 0$ and $u \big{|}_{\Omega'} \equiv 0$, then $u \equiv 0$ on a neighborhood of $x_{0}$.
\end{Teo}

Equipped with the theorem above, we are able to prove the following
\begin{Pro} \label{no solutions of Delta type}
Let $P$ be an operator satisfying \eqref{no singularities} and $u \in \mathcal{B}(\Omega)$ be a hyperfunction with support contained in $\{x_{0}\}$ such that $Pu = 0$. Then $u \equiv 0$.  
\end{Pro}
\begin{proof}
Let $u \in \mathcal{B}(\Omega)$ such that $\supp u \subset \left\{x_{0}\right\}$ and $Pu = 0$. Since \eqref{no singularities} holds, it follows that the polynomial $P_{m}(x_{0}, \xi)$ is not zero, which implies in particular the existence of $\xi_{0} \in \R^{n} \setminus \left\{0\right\}$ such that
\begin{equation*}
P_{m}(x_{0}, \xi_{0}) \neq 0.  
\end{equation*}
Let $\Phi: \Omega \to \R$ given by $\Phi (x) = \langle\xi_{0}, x - x_{0} \rangle $, where $\langle \cdot, \cdot \rangle$ denotes the usual inner product on the euclidean space. Then $\Phi$ satisfies all the conditions established on Theorem~\ref{Holmgren Theorem}; moreover, since $\supp u \subset \left\{x_{0}\right\}$, it is a consequence of the same result that 
\begin{equation*}
\text{$u \big{|}_{\Omega'} \equiv 0 \ \Rightarrow \ u \equiv 0$ on a neighborhood of $x_{0} \ \Rightarrow \ u \equiv 0$,} 
\end{equation*}
as we intended to prove. 
\end{proof}

By associating Proposition \ref{solution delta type} to Proposition \ref{no solutions of Delta type} we obtain the following
\begin{Teo} \label{hypo implies solvable}
 Let $P$ be a differential operator with real-analytic coefficients that does not degenerate at $x_{0}$. If $P$ is $\hypo(\D'_{(s)}, \D')$ on a neighborhood of $x_{0}$, then $\transp{P}$ is $\Cinf$-locally solvable at $x_0$. 
\end{Teo}

It is worth noting that for any $s > 1$ and $1 < s_{0} < s$, one has
\begin{equation*}
\D'_{(s)}(\Omega) \subset \D'_{s_{0}}(\Omega).
\end{equation*}
Therefore we have the following
\begin{Cor} \label{hypo gevrey romieu distributions}
Let $P$ be a differential operator with real-analytic coefficients that does not degenerate at $x_{0}$. If there exists $s > 1$ for which $P$ is $\hypo(\D'_{s}, \D')$ then $\transp{P}$ is $\Cinf$-locally solvable at $x_0$.
\end{Cor}

\begin{Exe}
We  recall one more time the \emph{Lewy's operator} (see Example \ref{Lewy}), which is not locally solvable even in $(\D', \Cinf)$ sense at the origin. Since $^{t}L = L$ it follows from Corollary~\ref{hypo gevrey romieu distributions} that
\begin{equation*}
\text{$L$ is \textbf{not} $\hypo(\D'_{s}, \D')$ for any $s > 1$, on any open neighborhood of the origin.}
\end{equation*}
An analogous result holds for the \emph{Mizohata's operators} (see Example \ref{Mizohata}) when $k$ is odd. 
\end{Exe}

\begin{Obs}
Let $P$ the ordinary differential operator defined by 
\begin{equation*}
P =  x \frac{\dd}{\dd x} + 1.  
\end{equation*} 
Then the Dirac delta is a solution for $P$, which shows that  without the hypothesis \eqref{no singularities} one can find solutions of delta type for an operator with real-analytic coefficients. The irregularity of $P$ (as defined in \eqref{definition irregularity}) is $1$, which  means that $P$ is $\hypo(\mathcal{B}, \D')$ (by \eqref{result  Komatsu}). In particular
\begin{equation*}
\text{$P$ is $\hypo(\D'_{(s)}, \D')$, for any $s>1$}.
\end{equation*}
Its transpose, however, is given by
  \begin{align*}
    \transp{P}= -x \frac{\dd }{\dd x},
  \end{align*}
  which is not $\Cinf$-locally solvable at $0$. Indeed, one cannot find $u$ smooth in a neighborhood of the origin satisfying $\transp{P} u=1$, since the right-hand side does not vanish at the origin. This example proves that the hypothesis of \emph{no degeneracy at $x_{0}$} is in fact necessary to obtain Theorem~\ref{hypo implies solvable}. 
\end{Obs}

\appendix 

\section{Infinite order differential operators } \label{Apendice 1}

In this section we discuss two types of infinite order differential operators acting on Gevrey spaces. These operators are morphisms in the Gevrey category and can be used to prove structural theorems for ultradistributions.

The two classes of operators in question are called \emph{ultradifferential operators of type $\{p!^{s}\}$}  and  \emph{ultradifferential operators of type $(p!^{s})$}. The former regards  morphisms for the Gevrey class of order $s$ of \emph{Roumieu type} and the latter  morphisms for the  Gevrey class of order $s$ of \emph{Beurling type}.

\begin{Def}\label{symbols ultradiff}[Symbols of ultradifferential operators]
	Consider an entire function in $\C^m$ given by
	\begin{align*}
		Q(\zeta)= \sum_{\alpha \in \Z_+^{m}} a_{\alpha} \zeta^{\alpha}.
	\end{align*}
	We shall say that $Q$ is the \emph{symbol of an ultradifferential operator of type $\{p!^{s}\}$} if, for every $\epsilon>0$, there exists $C_\epsilon>0$ such that
	\begin{align}\label{OpRoumieu}
		a_{\alpha}  &\leq  \frac{C_\epsilon \epsilon^{|\alpha|}}{|\alpha|!^{s}}, \quad \forall \alpha \in \Z^m_+.
	\end{align}
	Similarly, we shall say that $Q$ is the \emph{symbol of an ultradifferential operator of type $(p!^{s})$} if there exist $C>0$ and $h>0$ such that
	\begin{align}\label{OpBeurling}
		a_{\alpha}  &\leq  \frac{C h^{|\alpha|}}{|\alpha|!^{s}}, \quad \forall \alpha \in \Z^m_+.
	\end{align}
\end{Def}

	Condition \eqref{OpRoumieu} is equivalent to say that for every $\delta>0$ there is $C_\delta>0$ such that
	\begin{align}\label{OpRoumieu2}
		|Q(\zeta)|\leq C_\delta e^{\delta|\zeta|^{1/s}},
	\end{align}
 whilst condition \eqref{OpBeurling} is equivalent to the existence of constants $L>0$ and $r>0$ such that
	\begin{align}\label{OpBeurling2}
		|Q(\zeta)|\leq L e^{r|\zeta|^{1/s}}.
	\end{align}

\begin{Obs}
For the sake of brevity, from now on we will only consider the Gevrey space of Roumieu type. It will be clear how to adapt the results for the context of Beurling classes. 
\end{Obs}

Given a symbol of an ultradifferential operator $Q$ of type $\{p!^s\}$, we can define in ultradifferential operator of type $\{p!^{s}\}$, also denoted by $Q$, as it follows: let $U \subset \R^{n}$ be an open set and consider the map
\begin{align*}
	G^{s}(U) \ni u \mapsto  Q(\del_x) u(x)  \dfn \lim_{k \to  \infty} \sum_{|\alpha|\leq k} a_\alpha \del^{\alpha}_xu(x).
\end{align*}
 Note that $Q$ is a local operator, meaning that the germ of the Gevrey function $f\in G^{s}(U)$ at $x\in U$ determines the germ of $Qf$ at $x$. In particular,
\begin{equation*}
\text{$Q: G^{s}_{c}(U) \lra G^{s}_{c}(U)$ is a continuous map}.
\end{equation*}
Hence its transpose $\transp{Q}: \D'_{s}(U) \to \D_{s}'(U)$ is a continuous map as well. Since we are only considering infinite order operator with constant coefficients, it is not difficult to show that $\transp{Q}(\del_{x}) = Q(-\del_{x})$ when acting on Gevrey functions and, consequently, we can  extend  $Q$ to  ultradistributions of order $s$ by defining $Q(\del_x)= \transp Q(-\del_x)$.

\begin{Def} \label{Strongly elliptic}[Strongly elliptic operators]
	We shall say that  an ultradifferential operator $Q$  is \emph{strongly elliptic}  if there exist an open conic neighborhood $\Gamma$ of $\Rm$ in $\C^m$ and a constant $c>0$ such that
	\begin{align*}
		|Q(i\zeta)|\geq c, \ \text{for every $\zeta \in \Gamma.$}
	\end{align*}
 If $Q$ is strongly elliptic we will say that $Q$ satisfies \emph{condition $(\mathfrak{e})$}, meaning that its symbol satisfies the property described above for some constant $c$ and some cone $\Gamma$.
\end{Def}

It is worth to point out that the condition $(\mathfrak{e})$ in this text is weaker than similar properties established in \cite{cc} or \cite{rag}. As we shall see below, this property is important because it guarantees that the kernel of our operators only have analytic functions.

\subsection{Representations of ultradistributions in terms of Gevrey functions}

Next we provide a proof for a structural theorem, which shows that every ultradistribution of class $s$ can be represented by the action of an ultradifferential operator of type $\{p!^{s}\}$ on Gevrey functions of class $t$ with $t>s$. Moreover, we prove that the kernel of ultradifferential operators satisfying $(\mathfrak{e})$ contains only real-analytic functions. At the end of the subsection we state the results regarding operators of type $(p!^s)$ used in this work.

There is a special model of ultradifferential symbol which will be quite useful. First we fix a non-decreasing function $\sigma: (0, \infty) \lra (0, \infty)$ such that 
\begin{align}\label{paraoinfinitoealem}
	\lim_{x \to \infty} \sigma(x) = +\infty.  
\end{align}
Then we define
	\begin{align}\label{DefinicaodeQsigma}
		Q(\zeta)= \prod_{p=1}^{\infty} \bigg(1- \frac{\langle \zeta\rangle^{2}}{p^{2 s} \sigma(p)}\bigg).
	\end{align}

The next lemma first appeared in \cite{ragtese} and was strongly inspired by the study of hyperdifferential operators by Cordaro in \cite{cor} (see also \cite[Proposition $8.2.1$]{tre}).

\begin{Lem}\label{Operadoresdeordeminfinitaexistem}
  Assume that $\sigma:(0, \infty)\lra (0, \infty)$ satisfies \eqref{paraoinfinitoealem}, then $Q$ defined by \eqref{DefinicaodeQsigma} is symbol of an ultradifferential operator of type $\{p!^{s}\}$ satisfying $(\mathfrak{e})$.
\end{Lem}
\begin{proof}
        In fact, note that
	\begin{align} \label{estimate for Q 2}
		|Q(\zeta)|&\leq  \prod_{p=1}^{\infty} \bigg(1+ \frac{|\langle \zeta\rangle^{2}|}{p^{2 s} \sigma(p)}\bigg) \leq  \prod_{p=1}^{\infty} \bigg(1+ \frac{|\zeta|^{2/s}}{p^{2} (\sigma(p))^{1/s}}\bigg)^{s}.
	\end{align}
	In order to show that $Q$ satisfies \eqref{OpRoumieu2}. Given $\epsilon > 0$, take $q$ the first integer such that $\sigma (p)^{1/s} \geq \epsilon^{-2}$ whenever $p \geq q+1$. Then
	\begin{align}
	\prod_{p=1}^{\infty} \bigg(1+ \frac{|\zeta|^{2/s}}{p^{2} (\sigma(p))^{1/s}}\bigg)^{s} 
	&\leq C_{\epsilon} \prod_{p=1}^{\infty} \bigg(1+ \frac{\epsilon^{2} |\zeta|^{2/s}}{p^{2} }\bigg)^{s}, \label{estimate product Hadamard}
	\end{align}
where 
\begin{equation*}
C_{\epsilon} = \sup_{\zeta \in \C^{m}} \bigg[ \prod_{p=1}^{q} \bigg(1+ \frac{|\zeta|^{2/s}}{p^{2} (\sigma(p))^{1/s}}\bigg)^{s} \bigg(1+ \frac{\epsilon^{2} |\zeta|^{2/s}}{p^{2} }\bigg)^{-s}  \bigg].
\end{equation*}

Using the Weierstrass factorization 
	\begin{align*}
		\prod_{p=1}^{\infty} \bigg(1- \frac{z^{2}}{p^{2} }\bigg)=  \frac{\sin (\pi z)}{\pi z} , \quad z\in \C\setminus \{0\},
	\end{align*}
we obtain
\begin{align*}
\prod_{p=1}^{\infty} \bigg(1+ \frac{\epsilon^{2} |\zeta|^{2/s}}{p^{2} }\bigg)^{s} 
&= \bigg(\frac{e^{- \pi   \epsilon  |\zeta|^{1/s}} - e^{ \pi   \epsilon  |\zeta|^{1/s}}}{2 i (\pi i  \epsilon  |\zeta|^{1/s})} \bigg)^{s} \\
&= \bigg(\frac{e^{- \pi   \epsilon  |\zeta|^{1/s}} \big(e^{2 \pi   \epsilon  |\zeta|^{1/s}} - 1 \big)  }{2  (\pi  \epsilon  |\zeta|^{1/s})} \bigg)^{s} \\
&\leq \bigg(\frac{ e^{2 \pi   \epsilon  |\zeta|^{1/s}} - 1 }{2  (\pi  \epsilon  |\zeta|^{1/s})} \bigg)^{s}\\
 &\leq  e^{2 \pi s   \epsilon  |\zeta|^{1/s}}.
\end{align*}
By associating the equation above to \eqref{estimate product Hadamard}, we conclude that $Q$ indeed satisfies \eqref{OpRoumieu2}.

Next we verify that $Q$ satisfies condition $(\mathfrak{e})$ in the cone
\begin{align*}
	\Gamma= \{ \zeta= \xi + i \eta: |\eta| \leq |\xi|/2\}.  
\end{align*}
Note that, if $ \zeta \in \Gamma$,
\begin{align*}
	|Q(i \zeta)|&= \prod_{p=1}^{\infty} \bigg|1+ \frac{\langle \zeta\rangle^{2}}{p^{2s}\sigma(p)}\bigg|\geq \prod_{p=1}^{\infty}\bigg( 1+ \frac{\Re \langle \zeta\rangle^{2}}{p^{2s}\sigma(p)}\bigg) = \prod_{p=1}^{\infty}\bigg( 1+ \frac{|\xi|^{2} - |\eta|^{2}}{p^{2s}\sigma(p)}\bigg) \geq 1,
\end{align*}
which proves our claim. 
\end{proof}
  
   \begin{Obs}
By picking $\sigma$ as positive constant function, we have an interesting class of ultradifferential operators of type $(p!^s)$. 
\end{Obs}  

\begin{Teo} \label{prop a3}
	Consider $u\in \E'_s(U)$ an arbitrary ultradistribution with compact support; given $t>s$, there exist $Q$ an ultradifferential operator of type 
	$\{p!^s\}$ satisfying condition $(\mathfrak{e})$  and $f \in G^{t}(U)$ such that 
	\begin{equation*}
		u= Qf \  \text{in} \ \D'_s(U).
	\end{equation*}
\end{Teo}

\begin{proof}
Let $ u \in \E'_s( U)$; our goal is to choose an adequate increasing function $\sigma$ which will allow us to represent $u$ using the operator $Q$ described in Lemma \ref{Operadoresdeordeminfinitaexistem}.  For every $\epsilon>0$, there exists a constant $C_\epsilon>0$ such that
	\begin{align}\label{Desultradistribuicao2}
		| \hat{u}(\xi)|\leq C_\epsilon e^{\epsilon |\xi|^{1/s}}, \quad \forall \xi \in \R^{m}.
	\end{align}
Take a continuous, strictly decreasing function $C:  (0, \infty) \lra (0, \infty)$ such that
\begin{equation*}
C_{\epsilon} \leq C(\epsilon), \ \ \forall \epsilon > 0.
\end{equation*}
With no loss of generality we may assume that
\begin{equation*}
\text{$C(\epsilon)> 1$ if $\epsilon\leq 1$ and $\lim_{\epsilon \to 0^{+}} C(\epsilon) = +\infty$. }
\end{equation*}
In particular,
\begin{align}\label{NomeparaaCdeepsilon}
  | \hat{u}(\xi)|\leq C(\epsilon) e^{\epsilon |\xi|^{1/s}}, \quad \forall \xi \in \R^{m};
\end{align}

 Let $ a= \log C(1)>0$ and define 
 \begin{equation*}
 \text{$\gamma: [a, \infty) \lra (0,1]$ as the inverse of the function $\log C: (0,1]\lra [a, \infty)$.}
 \end{equation*}
 Pick $\mu \in (0,\frac{1}{2}]$ and define $\sigma: [0,\infty) \lra (0, \infty)$ by
	\begin{align*}
		\sigma(\rho)=
		\left\{\begin{array}{cl}
		\displaystyle 	\frac{\log 2}{6}  \frac{1}{ \rho^{-\mu}+ \gamma(\rho^{\mu})},& \textrm{ if } \rho^{\mu} \geq a;\\
		\displaystyle	\frac{\log 2}{6} \frac{1}{ a^{-\mu}+ \gamma(a^{\mu})},& \textrm{ if } 0\leq  \rho^{\mu}< a.
		\end{array}
		\right.
	\end{align*}
Hence $\sigma$ is a continuous and increasing function, which satisfies
        \begin{align*}
        \lim_{\rho \to + \infty} \sigma(\rho) = +\infty.  
        \end{align*}

	Finally, in order to find a Gevrey function $f$ of order $t$ as in the statement, we need to show that $Q$ has an inverse (in the sense of pseudodifferential operators). In order to do it, we need another inequality; for each fixed $\zeta \in \C^{m}$, take $q= q(\zeta)\in \Z_+$ such that
	\begin{equation} \label{intermediate inequality - part two}
		q^{2s} \sigma(q) \leq |\zeta|^{2} \leq (q+1)^{2s} \sigma(q + 1).
	\end{equation}
	When $|\zeta|$ is sufficiently large, one has $\sigma(q) \geq 1$. In this situation \eqref{intermediate inequality - part two} yields
	\begin{equation} \label{consequences intermediate inequality - part two}
		q \leq |\zeta|^{1/s} \ \text{and} \ 	q + 1 \geq \frac{|\zeta|^{1/s}}{\sigma(q + 1)^{1/2s}} \geq \frac{|\zeta|^{1/s}}{\sigma(|\zeta|^{1/s} + 1)^{1/2s}} \geq \frac{|\zeta|^{1/s}}{\sigma(|\zeta|^{1/s} + 1)}.
	\end{equation}
	Therefore, if $\xi \in \R^{m}$ has norm is sufficiently large so both \eqref{intermediate inequality - part two} and \eqref{consequences intermediate inequality - part two} hold,  it follows that
	\begin{align} \label{estimate for Q}
		|Q(i \xi)|  &\geq \prod_{p=1}^{q}\bigg(1 + \frac{ |\xi|^{2}}{p^{2s}\sigma(p)}\bigg)\geq 2^{q}\geq \frac{1}{2} \exp\left( \log 2 \frac{ |\xi|^{1/s}}{\sigma(|\xi|^{1/s}+1)}\right).
	\end{align}
	By combining to \eqref{NomeparaaCdeepsilon} to \eqref{estimate for Q}, one infers that
	\begin{align*}
		\bigg| \frac{ \hat{u}(\xi)}{Q(i\xi)}\bigg| &\leq  C(\epsilon) 2\exp\bigg(\epsilon| \xi|^{1/s}-\log 2 \frac{|\xi|^{1/s}}{\sigma(|\xi|^{1/s}+1)} \bigg)\\
		&\leq  2 \exp\bigg(\log C(\epsilon)+ \epsilon|\xi|^{1/s}- \log 2 \frac{|\xi|^{1/s}}{\sigma(|\xi|^{1/s}+1)} \bigg).
	\end{align*}
	
If we further assume that $|\xi|$ is large enough so $ a<(|\xi|^{1/s}+1)^{\mu}$,
and if we choose $\epsilon= \gamma((|\xi|^{1/s}+1)^{\mu})$, then $\log C(\epsilon) = (|\xi|^{1/s}+1)^{\mu}$ and
	\begin{align*}
		\bigg| \frac{\hat{u}( \xi)}{Q(i\xi)}\bigg|  &\leq 2 \exp\bigg((|\xi|^{1/s}+1)^{\mu}+ \epsilon|\xi|^{1/s}-  6 |\xi|^{1/s} \left[ (|\xi|^{1/s}+1)^{-\mu}+ \epsilon \right] \bigg)
		\\    
		&\leq  2 \exp\bigg((|\xi|^{1/s}+1)^{\mu} -  \frac{6 |\xi|^{1/s}}{(|\xi|^{1/s}+1)^{\mu} } \bigg)\\
		&\leq  2 \exp\bigg( \frac{2|\xi|^{2\mu/s}-  3 |\xi|^{1/s}}{|\xi|^{\mu/s}} \bigg)\\
		&\leq  2 \exp\bigg( \frac{2|\xi|^{1/s}-  3 |\xi|^{1/s}}{|\xi|^{\mu/s}} \bigg)\\
		&\leq  2 \exp\big( -   |\xi|^{(1-\mu)/s} \big).
	\end{align*}
	
Given $t>s$, let $\mu = \displaystyle \min  \left\{ \frac{1}{2}, 1- \frac{s}{t} \right\}$; then there exists $C>0$ such that
\begin{align}\label{DesigualdadeFundamentalpararepresentacao}
\bigg| \frac{\hat{u}(\xi)}{Q(i\xi)}\bigg|     \leq  C e^{-  |\xi|^{1/t} }, \quad \forall \xi \in \R^m,
\end{align}
which implies that
	\begin{align*}
		f_\delta(x)= \frac{1}{(2\pi)^{m}} \int_{\R^m} \frac{\hat{u}(\xi)}{Q(i\xi)}e^{i x \xi- \delta|\xi|^{2}} \dd \xi, \ \ \delta > 0, 
	\end{align*}
	is a uniformly bounded family of functions in $G^{t}(\R^m)$. Since $G^{t}(\R^m)$ is a Montel space, there exists a subsequence $\delta_{k}$ converging to $0$ for which $f_{\delta_{k}}$ converges in $G^{t}(\R^m)$. Let us call $f$ its limit.
	
	Notice that
	\begin{align*}
	Q f_\delta(x) = \frac{1}{(2\pi)^{m}} \int_{\R^m} \hat{u}(\xi)e^{i x \xi- \delta|\xi|^{2}} \dd \xi.  
	\end{align*}
	Thus, on the one hand, $Q f_\delta$  converges to $u$ in $\D'_s(\R^m)$ when $\delta \to 0$ and, on the other hand, $Q f_{\delta_{k}}$ converges to $Qf$ in $\D'_s(\R^m)$, which allows us to conclude that
	\begin{align*}
		Q f= u \ \text{in $\D'_s(\Rm)$},
	\end{align*}
 finalizing the proof. 
\end{proof}

\begin{Obs}\label{representatudo}
  By analyzing the last step of the proof one can see that $Q$ can be used not only to represent $u$ but any other ultradistribution that satisfies \eqref{DesigualdadeFundamentalpararepresentacao}. In particular, it is not difficult to check that every distribution with compact support and even  every ultradistribution of order $s_0< s$ with compact support can be represented by the same $Q$.
  
  Moreover, given  $u_1, \ldots, u_\ell\in \E'_s(\R^m)$ a family of compactly supported ultradistributions, we may define $C:(0,\infty) \lra (0,\infty)$ satisfying  \eqref{NomeparaaCdeepsilon} for all the ultradistributions and the proof above provides an ultradifferential operator of type $\{p!^{s}\}$, $Q$, and $f_1,\ldots, f_\ell\in G^{t}(\R^m)$ such that $u_j= Qf_j$, for any $j \in \left\{1, \ldots, \ell \right\}$.
\end{Obs}

Until now, condition $(\mathfrak{e})$ has not played any role. However it is important because it guarantees that the kernel of an ultradifferential operator contains only real-analytic functions. Our next result is an adaptation of \cite[Lemma $1.1$]{cor}.

\begin{Teo}\label{propdokernel}
	Let $Q$ be an ultradifferential operator of type $\{p!^s\}$  that satisfies condition $(\mathfrak{e})$ and  fix $V$ an open set in $\R^m$. If $v\in \Cinf(V)$ satisfies $Qv=0$, then $v\in \Com(V)$.  
\end{Teo}
\begin{proof}
Since the result has a local nature, we may assume that $V$ is a bounded open set and fix $p_0 \in V$.	Denote by $y$ the variables on $V$ and pick  $v\in \Cinf(V)$ satisfying $Q(\del_y)v=0$. Then
	\begin{equation} \label{identity transpose}
		\int_V v(y) [Q(-\del_y)\psi(y)] \dd y=0, \quad \forall  \psi\in G^{s}_c(V).
	\end{equation}
Let $V_1\subset\subset V$ be a relatively compact open neighborhood of $p_0$ and fix $\varphi\in G^{s}_c(V)$ such that  $\varphi \equiv 1$ on $V_1$. Notice that, for any $H\in \mathcal{O}(\C^m)$,
	\begin{equation} \label{chain rule consequence}
		Q(-\del_y)[\varphi H](y)- \varphi Q(-\del_y) H(y)= \sum_{\alpha \geq 0} (-1)^{|\alpha|} a_\alpha \sum_{\beta\leq \alpha, |\beta|>0} { \alpha \choose \beta} \del_y^{\beta} \varphi(y) \del_y^{\alpha-\beta} H(y).
	\end{equation}

	Let $\Gamma$ be a conic neighborhood of $\R^m$ where condition $(\mathfrak{e})$ holds.  Pick $\lambda \in  \big(0, \frac{1}{4} \big)$ such that $\Gamma_\lambda\subset \subset  \Gamma$, where
	\begin{align*}
		\Gamma_\lambda = \{ \zeta: |\Im \zeta|< \lambda |\Re \zeta|\}.    
	\end{align*}
Notice that $K= \overline{V}\setminus V_1$ is a compact set; for $\delta > 0$,  let 
	\begin{equation*}
\text{$D_{\delta}= \{ z\in \C^m: d(z, p_0)< \delta\}$, \ \ \ $K_\delta= \{ z\in \C^m: d(z, K)\leq \delta\}.$}
	\end{equation*}
By choosing $\delta$ sufficiently small, we may assume that $d(D_{\delta},K_\delta)> 3 \delta/\lambda$. From now on, we fix a $\delta \in (0,1)$ satisfying such condition.  

It follows from  \eqref{identity transpose} and \eqref{chain rule consequence} that
	\begin{align}
		\left|\int_V v(y) \varphi(y) Q(-\del_y) H(y) \dd y \right| &= \bigg|\int_V v(y)\big[ Q(-\del_y)[\varphi H](y)- \varphi Q(-\del_y) H(y)\big]\dd y\bigg|\nonumber \\
		&\leq  \sum_{\alpha \geq 0} | a_\alpha| \sum_{\beta\leq \alpha, |\beta|>0} { \alpha \choose \beta} \bigg|\int_{V\setminus V_1} v(y)\del_y^{\beta}  \varphi(y) \del_y^{\alpha-\beta} H(y)\dd y\bigg|. \label{integral estimate v}
	\end{align}  

	Since $Q$ is of class $\{p!^{s}\}$, for every $\epsilon>0$ there exists $C_\epsilon>0$ for which
	\begin{align*}
		|a_\alpha|\leq C_\epsilon \epsilon^{|\alpha|} \frac{1}{|\alpha|!^{s}}.
	\end{align*}
	From the fact that $\varphi\in G^{s}_c(V)$ it follows the existence of $C, h>0$  such that
	\begin{align*}
		|\del_y^{\beta} \varphi(y)|\leq C h^{|\beta|} \beta!^{s}, \ \ \forall y \in \R^{m}.
	\end{align*}
	Furthermore, one deduces from Cauchy estimates that
	\begin{align*}
		|\del_y^{\alpha} H(y)|\leq \frac{|\alpha|!}{\delta^{|\alpha|}}\sup_{K_\delta} |H|, \quad \forall y \in K.
	\end{align*}

	Define $C_v= \displaystyle \int_K |v(y)| \dd y$; by applying the three inequalities above to \eqref{integral estimate v}, we deduce that
	\begin{align*}
	\left|\int_V v(y) \varphi(y) Q(-\del_y) H(y) \dd y \right| &\leq  \sum_{|\alpha| \geq 0} | a_\alpha| \sum_{\beta\leq \alpha \atop |\beta|>0} { \alpha \choose \beta} \int_{V\setminus V_1}|v(y)||\del_y^{\beta} \varphi(y)| |\del_y^{\alpha-\beta} H(y)|\dd y\\
		&\leq \big(  \sup_{K_\delta} |H| \big) C_v C_\epsilon \sum_{|\alpha| \geq 0}  \epsilon^{|\alpha|}  \sum_{\beta\leq \alpha \atop |\beta|>0} \frac{|\alpha- \beta|!}{\delta^{|\alpha-\beta|}} \frac{1}{|\alpha|!^{s}}  { \alpha \choose \beta}  C h^{|\beta|} \beta!^{s}\\
		&\leq \Big(  \sup_{K_\delta} |H| \Big)  C_v C  C_\epsilon\sum_{|\alpha| \geq 0}   \bigg(\frac{(1+ h)\epsilon}{\delta}\bigg)^{|\alpha|}. 
	\end{align*}  
	Thus, if $\epsilon$ is taken sufficiently small, it follows the existence of  $\widetilde{C}_\delta > 0$ such that
	\begin{equation} \label{estimate integral v II}
		\left|\int_V v(y) \varphi(y) Q(-\del_y) H(y) \dd y \right|\leq  \widetilde{C}_\delta \Big(  \sup_{K_\delta} |H| \Big). 
	\end{equation}
	
	Now consider
	\begin{align*}
		w \mapsto H_\rho(z, w)= \frac{1}{(2\pi)^{m}} \int_{\R^m} \frac{1}{Q(-i\xi)} e^{i(z-w)\xi-\rho |\xi|^{2}} \dd \xi,
	\end{align*}
	where $z$ is a parameter in $D_{\delta}$ and $\rho>0.$ Note that
	\begin{align*}
		Q(- \del_y) H_{\rho}(z, y)= \frac{1}{(2\pi)^{m}} \int_{\R^m}  e^{i(z-y)\xi-\rho |\xi|^{2}} \dd \xi.
	\end{align*}
Therefore it is a consequence of \eqref{estimate integral v II} that
	\begin{align*}
		\bigg| \frac{1}{(2\pi)^{m}} \int_{\R^m}  e^{iz\xi-\rho |\xi|^{2}} \widehat{\varphi v}(\xi)   \dd \xi\bigg|&=     \bigg|\int_V v(y) \varphi(y) \frac{1}{(2\pi)^{m}} \int_{\R^m}  e^{i(z-y)\xi-\rho |\xi|^{2}} \dd \xi\dd y\bigg|\\
		&=     \bigg|\int_V v(y) \varphi(y) Q(-\del_y) H_{\rho}(z,y)\big]\dd y\bigg|\\
		&\leq  \widetilde{C}_\delta \sup_{y \in K_\delta} |H_{\rho}(z,y)|. 
	\end{align*}
	
It only remains to prove the existence of $\widetilde{C}>0$ such that for every  $z\in D_{\delta}$ and  $\rho\in (0,1)$, it holds that
\begin{align*}
\sup_{y \in K_\delta} |H_{\rho}(z,y)|< \tilde{C}.
\end{align*}
In fact, by Montel Theorem this would imply that $\varphi v \big{|}_{D_{\delta} \cap \R^{m}}$ is the restriction of a holomorphic function. Since $\varphi \equiv 1$ on $V_{1}$, we would deduce that $v \big{|}_{D_{\delta} \cap V_{1}}$ is the restriction of a holomorphic function and therefore it would be real-analytic in a neighborhood of $p_0$. 
	
With such goal, we apply the following deformation: $\xi \mapsto \zeta(\xi)= \xi+ i\lambda \frac{\Re (z-w)}{|\Re (z-w)|}|\xi|$. In this case, it follows that
	\begin{align*}
		H_\rho(z, w)
		&= \frac{1}{(2\pi)^{m}} \int_{\zeta(\R^m)} \frac{1}{Q(-i\zeta)} e^{i(z-w)\zeta-\rho \langle \zeta\rangle^{2}} \dd \zeta.
	\end{align*}
Next we estimate the terms inside the integral;	note first that $|Q(i\zeta)|>c$ for $\zeta \in \Gamma$.  Furthermore,
	\begin{align*}
		\Re \{i(z-w)\zeta-\rho \langle \zeta\rangle^{2}\}&= -\Im (z-w)\xi - \lambda|\Re(z-w)| |\xi|- \rho\bigg(|\xi|^{2}- \lambda^{2}|\xi|^{2}\bigg)\\
		&\leq -\Im (z-w)\xi - \lambda|\Re(z-w)| |\xi|.
	\end{align*}
Using that $z\in D_{\delta}$ and $w\in K_\delta$, we infer that $|\Im(z-w)|\leq 2\delta$ and $|\Re(z-w)|\geq \displaystyle \frac{3\delta}{\lambda} - 2\delta$, using this we conclude
	\begin{align*}
		\Re \{i(z-w)\zeta-\rho \langle \zeta\rangle^{2}\}\leq -\frac{\delta|\xi|}{2}.
	\end{align*}
	This inequalities yields that $|H_\rho(z,w)|$ is uniformly bounded for any $z \in D$, $w\in K_\delta$ and $\rho\in (0,1)$. As we desired.
\end{proof}

The previous result can be enhanced by applying the  representation theorem for ultradistributions.

\begin{Cor}\label{Cordokernelultradist}
	Let $Q$ an ultradifferential operator of type $\{p!^s\}$  that satisfies condition $(\mathfrak{e})$ and fix $V$ an open set in $\R^m$. If $v\in \D'_s(V)$ satisfies $Qv=0$, then $v\in \Com(V)$.  
\end{Cor}
\begin{proof}
	Given $v \in \D'_s(V)$ for which $Qv=0$, fix $p_0 \in V$; we intend to show that $v$ is real-analytic in a neighborhood of $p_0$. Let $\chi\in G^{s}_c(V)$ such that $\chi \equiv 1$ on an open neighborhood $W$ of $p_0$.
	
	It follows from Theorem~\ref{prop a3} the existence of an ultradifferential operator  $J$ of type $\{p!^s\}$ satisfying condition $(\mathfrak{e})$ and $f\in \Cinf(V)$ such that
	\begin{align*}
		J f= \chi v \ \text{in $\D'_s(V)$}.
	\end{align*}
	Hence
	\begin{align*}
		QJ f= Q u=0 \ \text{on $W$.}
	\end{align*}

	 Notice that the composition $Q \circ J$ is made of two operators of type $\{p!^{s}\}$ which satisfy condition $(\mathfrak{e})$, which allows us to deduce that it is an ultradifferential operator of same type. Theorem~\ref{propdokernel} implies that $f\in \Com(W)$; since $J$ is continuous when restricted to $\Com(W)$, it follows that $v=Jf\in \Com(W),$ as we intended to prove.
       \end{proof}

       \begin{Obs}\label{remarkGrVaiemGr}
       Pick $r\in [1, s)$ and consider $g \in G^{r}(U)$ for which there is $v \in \D'_s(U)$ and $Q$ an ultradifferential operator of type $\{ p!^{s}\}$ such that $Qv= g$ than we have that $v\in G^{r}(U)$. Indeed, fix $x_0\in U$ and $\chi\in G^{t}_c(U)$ such that $\chi=1$ in a neighborhood of $x_0$. Define
       \begin{align*}
         f_\delta(x)= \frac{1}{(2\pi)^{m}} \int_{\R^m} \frac{ \widehat{\chi g}(\xi)}{Q(i \xi)} e^{i x \xi- \delta |\xi|^2} \dd \xi
       \end{align*}
       is a uniformly bounded family of functions in $G^r(\R^m)$ and thus there is a subsequence $f_{\delta_k}$ converging to a function $f$ in $G^r(\R^m)$. Note that
       \begin{align*}
         Q f_\delta(x) = \frac{1}{(2\pi)^m} \int_{\R^m} \widehat{\chi g}(\xi) e^{i x \xi- \delta |\xi|^2} \dd \xi
       \end{align*}
       converges to $\chi g$ in $G^{r}(\Rm)$ by the continuity of $Q$ in $G^{r}(\Rm)$ it follows that $Qf= \chi g$. 
       Therefore, it follows that
       \begin{align*}
        \big(Q(v-f)\big)|_{V}= (\chi Qv -Qf) |_V=0.
       \end{align*}
       By Theorem~\ref{Cordokernelultradist} it follows that $v-f\in C^{\omega}(V)$ and therefore $v$ restricted to a neighborhood of $x_0$ belongs to $G^{r}$; since $x_0$ is arbitrary, we conclude that $v\in G^{r}(U)$.
       
        This result can also be adapted by ultradifferential operators defined on locally integrable structures used in Section~\ref{LIS Gevrey}; in fact, if we are dealing with a Gevrey locally integrable structure of order $s_0$, then if $1< s_0 < r< s$, it is a consequence of \cite[Proposition 7.2]{rag} that the kernel of $Q_s$ only has Gevrey functions of order $s_0$ thus we can adapt the argument above and the Baouendi-Treves approximation theorem to prove that if $g\in G^{r}(U)$ and $Q_s v= g$, then $v\in G^{r}(U)$. 
       \end{Obs}

\begin{Obs}\label{RepresentacaodeHiperfuncoes}
  Similar results also hold true in the context of hyperfunctions. For instance,  it was proved that in \cite[Theorem 1.1]{cor} that if $u\in \mathcal{B}(U)$, then given $x_{0}\in U$ one may find a hyperdifferential operator $Q$, i.e., an ultradifferential operator of type $\{p!\}$, satisfying $(\mathfrak{e})$, $V \subset\subset U$ a neighborhood of $p$ and $f\in \Cinf(V)$ such that $u= Q f$ on $V$. Furthermore, by adapting  the proof of Theorem~\ref{prop a3}, one can prove the following version of Cordaro's result:
  \end{Obs}
  
	\begin{Teo} \label{hiperteo}
		Consider $u\in \mathcal{B}(\R^n)$. Given $s>1$ and $U$ a bounded open subset of $\R^n$, there exist $Q$ an hyperdifferential operator of type 
		$\{p!\}$ satisfying condition $(\mathfrak{e})$  and $f \in G^{s}(U)$ such that 
		\begin{equation*}
			u= Qf \  \text{in} \ \mathcal{B}(U).
		\end{equation*}
	\end{Teo}
	The main idea is to associate the hyperfunction $u$ to an analytic functional $\mu$ carried by a compact neighborhood of $U$ and use the continuity of $\mu$ to obtain a continuous strictly decrease function $\epsilon \mapsto C(\epsilon)$.  
	For more details on how to adapt Theorem~\ref{prop a3} to hyperfunctions one can see \cite{cor} and \cite{Kan}.  It is also not difficult to check that we have a result analogous to Remark~\eqref{representatudo}.

\begin{Obs} \label{kernel analytic case}
 Corollary~\ref{Cordokernelultradist} may be adapted as well to the hyperfunction setting. That is, if $Q$ is a hyperdifferential operator of type $\{p!\}$ that satisfies condition $(\mathfrak{e})$, then  every $u\in \mathcal{B}(V)$ such that $Qu =0$ is in fact real-analytic, see \cite[Theorem 1.2]{cor}.
\end{Obs}

In a similar fashion,  a structural theorem can be proved using ultradifferential operators of type $(p!^{s})$, where an ultradifferential operator can be used to represent all ultradistributions of order $s$. For our purpose, we need a representation theorem of Roumieu ultradistributions in terms of operators of type $(p!^{s})$. Since operators of type $(p!^s)$  are discontinuous when acting on $\D'_s(\R^m)$, we need to reinterpret this representation in a broader space, for instance, in the space of ultradistributions of order $s_0$ with $s_0< s$. 

\begin{Pro}\cite[Theorem 6.2]{rag} \label{proposition a4}
	There exists  an infinite order operator $Q$, of type $(p!^{s})$ and satisfying condition $(\mathfrak{e})$ such that, for every ultradistribution with compact support $u\in \E'_s(U)$, there exists $f\in G^{s}(U)$ such that $u=Qf$ in $\D'_{s_0}(U)$, for any $s_0<s$.
\end{Pro}


It is not difficult to check that kernel of ultradifferential operators of type $(p!^{s})$ satisfying condition $(\mathfrak{e})$ is also a subspace of the space of real-analytic functions. Indeed, an ultradifferential operator $Q$ of type $(p!^{s})$ is an ultradifferential operator of type $\left\{p!^{s_{0}}\right\}$, for any $s_{0} < s$, so Corollary~\ref{Cordokernelultradist} still holds for $Q$ and so we can state the following result:

\begin{Teo}\label{PropkernelultradistBeurling}
	Let $Q$ be an ultradifferential operator of type $(p!^s)$  that satisfies condition $(\mathfrak{e})$ and fix $V$ an open set in $\R^m$. If $v\in \D'_s(V)$ satisfies $Qv=0$, then $v\in \Com(V)$.  
\end{Teo}

\subsection{The action of ultradifferential operators on smooth functions}

In this subsection we intend to deal with a slightly different problem. Let  $u$ be an arbitrary element of $\Cinf(U)\setminus G^{s}(U)$; we intend to prove the existence of an infinite order operator $Q$ of type $\{p!^s\}$ for which
\begin{equation*}
Qu \in \D'_s(V)\setminus \D'(V), \ \ \text{for an  open set $V$ relatively compact in $U$.}
\end{equation*}

Before we proceed to its statement and proof, some previous steps are necessary. If $v \in G^{s}_c(U)$,  there exists $C>0$ such that 
\begin{align*}
	|\hat{v}(\xi)|\leq C^{N+1} N!^{s} |\xi|^{-N}, \ \ \forall N \in \Z_{+}, \ \  \forall \xi\in \R^{m}\setminus\{0\}.
\end{align*}
Hence, if $u \in \Cinf_c(U)\setminus G^{s}_c(U)$, for every $M\in \Z_+$ there exists $N_M \in \Z_+$ and $\xi_M\in \R^{m}\setminus\{0\}$ such that
\begin{align} \label{desigualdadedoUpSideDown}
	|\hat{u}(\xi_M)|> M^{N_M+1} N_M!^{s} |\xi_M|^{-N_M}.
\end{align}

Two assumptions can be made on the pairs $(N_M, \xi_M)$ in \eqref{desigualdadedoUpSideDown}. First note that  the map $\xi  \mapsto |\hat{u}(\xi)|$ is continuous; thus for every $R>0$ there exists $C_R>0$ such that
\begin{align*}
	|\xi|^{N_M}|\hat{u}(\xi)|\leq R^{N_{M}} C_R, \ \ \text{for every $\xi \in B_{R}(0)$}.
\end{align*}
Therefore if $M>0$ is such that $M > \max \left\{C_{R}, R\right\}$, we have
\begin{equation*}
 M^{N_M+1} N_M!^{s} \geq M^{N_M+1} =  M M^{N_M} > C_{R} R^{N_{M}},
\end{equation*}
which implies that $|\xi_M|> R$. 

Hence one may extract a strictly increasing sequence $M_j\in \Z_+$  such that
\begin{equation*}
\text{$\xi_{j}\doteq \xi_{M_j}$ satisfies \eqref{desigualdadedoUpSideDown} and $|\xi_{j}|> j^{j} (j+1)!^{s}, \ \ \forall j \in \Z_{+}$.}
\end{equation*}
Moreover, suppose for a moment that the sequence $N_j\doteq N_{M_j}$ is bounded by  some natural number $N_0$. Since $u\in \Cinf_c(U)$, there exists $C_{N_0}>0$ such that
\begin{align*}
	|\hat{u}(\xi)|\leq C_{N_0}(1+|\xi|)^{-N_0-1}, \ \ \forall \xi \in \R^{m}.
\end{align*}
Hence
\begin{align*}
	|\xi_{j}|^{-N_j}<  M^{N_j+1} N_j!^{s} |\xi_{j}|^{-N_j}\leq  C_{N_0}(1+|\xi_{j}|)^{-N_0-1}, \ \ \forall j \in \Z_{+},
\end{align*}
which yields
\begin{align*}
	1<  C_{N_0}(1+|\xi_{j}|)^{-N_0-1}|\xi_{j}|^{N_j}\leq C_{N_0}(1+|\xi_{j}|)^{-1},
\end{align*}
contradicting the fact that $|\xi_{j}| \lra \infty$. Thus one may assume that 
\begin{equation*} 
\text{$N_j\in \Z_+$ is a strictly  increasing sequence and  $\lim_{j \to \infty} N_j = +\infty$}.
\end{equation*}
We are prepared to prove the following

\begin{Teo}\label{PropA1}
	For every $u\in \Cinf(U) \setminus G^{s}(U)$ there exists $Q$ an ultradifferential operator of class $\{p!^s\}$ and such that $Qu \in \D'_s(V) \setminus \D'(V)$, where $V\subset U$ is open.  
\end{Teo}

\begin{proof}
	Since $u\in\Cinf(U) \setminus G^{s}(U)$, there exists $\chi\in G^{s}_c(U)$   for which $\chi u\in \Cinf_{c}(U) \setminus G^{s}_{c}(U)$. By the previous arguments, there  sequences $\left\{M_j \right\}_{j \in \N}, \left\{N_j \right\}_{j \in \N} \subset \Z_+$ and $\left\{\xi_j \right\}_{j \in \N}\ \subset  \R^{m}$  such that 
	\begin{itemize}[leftmargin=*]
		\item $\left\{M_j \right\}_{j \in \N}, \left\{N_j\right\}_{j \in \N}$ are strictly increasing  and $\lim M_{j} = \lim N_{j} = + \infty$;
		\item $|\xi_j|> j^j (j+1)!^{s}$, for each $j \in \Z_{+}$;
		\item $|\hat{u}(\xi_{j})|> M_j^{N_j+1} N_j!^{s} |\xi_{j}|^{-N_j}$, for each $j \in \Z_{+}$.
	\end{itemize}

	Next we construct an infinite order operator whose symbol is given by
	\begin{align*}
		Q(\zeta)= \sum_{k\in \Z_+} (-1)^{k}a_{2k} \langle \zeta\rangle^{2k},
	\end{align*}
	where each $a_{2k}\geq 0$ will be chosen as it follows. Note that, for each $\xi \in \R^{m}$, 
	\begin{align*}
		|  \widehat{Q(\del_x)(\chi u)}(\xi)|&= |  Q(\del_x)(\chi u)_x(e^{-ix\xi})|\\
		&= \Big|  \int_{\R^m}\chi(x) u(x) Q(i\xi) (e^{-ix\xi}) dx\Big|\\
		&= | Q(i\xi)| \Big|  \int_{\R^m}\chi(x) u(x) e^{-ix\xi} dx\Big|\\
		&=  | Q(i\xi)||\widehat{\chi u}(\xi)|.
	\end{align*}
Since it is being assumed that $a_{2k} \geq 0$, we have
	\begin{align*}
		Q(i\xi)= \sum_{k\in \Z_+} (-1)^{k} a_{2k} \langle i\xi \rangle^{2k} = \sum_{k\in \Z_+}  a_{2k} |\xi|^{2k} \geq a_{2\ell} |\xi|^{2\ell} , \ \ \forall \ell \in \Z_{+}, \ \forall \xi \in \R^{m}.
	\end{align*}
Combining this inequality  with \eqref{desigualdadedoUpSideDown}, we deduce that
	\begin{equation} \label{Fourier coefficients of Qu}
		| Q(i  \xi_{j})||\widehat{\chi u}(\xi_{j})|          > a_{2k} |\xi_{j}|^{2k}  M_{j}^{N_{j}+1} N_{j}!^{s} |\xi_{j}|^{-N_{j}}, \ \ \forall j \in \Z_{+}.
	\end{equation}
	For every $j\in \Z_+$, define
	\begin{align*}
		\tau_j= \left\{
		\begin{array}{cc}
			N_j+j& \textrm{ if } N_j+j \textrm{ is even};\\
			N_j+j+1& \textrm{ if } N_j+j \textrm{ is odd}.\\
		\end{array} \right.
	\end{align*}
Moreover, let $a_0=1$ and
	\begin{align*}
		a_{2k}=\left\{
		\begin{array}{cl}
			(M_j^{N_{j}} j^{j}j!^{s}  N_j!^{s})^{-1 },    &\textrm{ if } 2k= N_j+j \textrm{ for some } j;\\
			(M_j^{N_{j}+1} j^j (j+1)!^{s} N_j!^{s})^{-1 },    &\textrm{ if } 2k= N_j+j+1 \textrm{ for some } j;\\
			0, & \textrm{ otherwise.}
		\end{array}\right.
	\end{align*}

Given $\epsilon>0$, consider $M_\epsilon\in \Z_+$ the first element $M_{j_{0}}$ of $\left\{M_{j}\right\}_{j \in \N}$ which satisfies the inequality $M_{j_{0}} \geq \displaystyle \frac{2^{s}}{\epsilon}$. If $\ell$  is such that $2k= \tau_\ell$, $M_{\ell}  \geq M_\epsilon$ and $\ell \geq M_\epsilon$, it follows that either
	\begin{align*}
		a_{2k} =  \frac{1}{(M_\ell^{N_{\ell}} \ell^{\ell} \ell!^{s}  N_\ell!^{s})} \leq \frac{(\epsilon/2^{s})^{\ell + N_{\ell}}}{(\ell!^{s}  N_\ell!^{s})} = \frac{(\epsilon/2^{s})^{\ell + N_{\ell}}}{(\ell + N_\ell)!^{s}} \left(\frac{(\ell + N_\ell)!}{\ell!  N_\ell!}\right)^{s} \leq  \frac{\epsilon^{2 k}}{(2k)!^{s}} 
	\end{align*}
or
\begin{align*}
	a_{2k} =  \frac{1}{(M_\ell^{N_{\ell} + 1} \ell^{\ell} (\ell+1)!^{s}  N_\ell!^{s})} \leq \frac{(\epsilon/2^s)^{\ell + N_{\ell} + 1}}{((\ell+1)!^{s}  N_\ell!^{s})}  \leq \frac{\epsilon^{\ell + N_{\ell} + 1}}{(\ell + N_\ell + 1)!^{s}} \leq  \frac{\epsilon^{2 k}}{(2k)!^{s}}.
\end{align*}

	Since $a_k=0$ whenever  $a_{k} \notin \left\{\tau_j\right\}_{j \in \Z_{+}}$ and there exists a finite number of positive integers for which $M_j\leq M_\epsilon$, for every $\epsilon > 0$ it is possible to find $C_\epsilon$ such that
	\begin{align*}
		a_{2k}  &\leq C_\epsilon \frac{\epsilon^{ 2k}}{(2k)!^{s}}, \quad \forall k \in \Z_+.
	\end{align*}
	This implies that $Q$ is of class $\{p!^s\}$. 

	It remains then to analyze how $Q$ acts on $u$. It is a consequence of \eqref{Fourier coefficients of Qu} that
	\begin{align*}
		|  \widehat{Q(\del_x)(\chi u)}(\xi_{j})| & > a_{2k}   M_j^{N_{j}+1} N_{j}!^{s} |\xi_{j}|^{2k-N_{j}};
	\end{align*}
consider now $2k$ being equal $N_{j} + j$ or $N_{j} + j + 1$, depending on the parity of the term $N_{j}+j$. Then either 
\begin{align*}
	|  \widehat{Q(\del_x)(\chi u)}(\xi_{j})| & > (M_j^{N_{j}} j^{j}j!^{s}  N_j!^{s})^{-1 }  M_j^{N_{j}+1} N_{j}!^{s} |\xi_{j}|^{2k-N_{j}} \\
	&> \frac{M_{j}|\xi_{j}|^{j}}{j^{j} j!^{s}} 
\end{align*}
or
\begin{align*}
	|  \widehat{Q(\del_x)(\chi u)}(\xi_{j})| & > 	(M_j^{N_{j}+1} j^j (j+1)!^{s} N_j!^{s})^{-1 }  M_j^{N_{j}+1} N_{j}!^{s} |\xi_{j}|^{2k-N_{j}} \\
	&> \frac{|\xi_{j}|^{j+1}}{j^{j} (j+1)!^{s}}.
\end{align*}
Using both cases, we are able to infer that
\begin{equation} \label{asymptotic growth Qu}
	|  \widehat{Q(\del_x)(\chi u)}(\xi_{j})| \geq \frac{|\xi_{j}|^{j}}{j^{j} (j+1)!^{s}}, \ \ \ \forall j \in \Z_{+}.
\end{equation}

	Suppose for a moment that $Q(\del_x)(\chi u)$ is a distribution with compact support. Then there exist $N_0\in \Z_+$ and $C_{N_0}>0$ such that
	\begin{align*}
		|  \widehat{Q(\del_x)(\chi u)}(\xi)|\leq C_{N_0}(1+|\xi|)^{N_0-1}.
	\end{align*}
	This inequality, associated to \eqref{asymptotic growth Qu}, leads us to conclude that
	\begin{align*}
		\frac{|\xi_{j}|^{j}}{j^{j}(j+1)!^{s}}  < C_{N_0}(1+|\xi_{j}|)^{N_0-1}, \ \ \forall j \in \Z_{+}.
	\end{align*}
Since we have from the beginning that $|\xi_j|\geq j^j (j+1)!^{s}$, it follows that
	\begin{align*}
		1 <  C_{N_0} |\xi_{j}|^{-j-1}(1+|\xi_{j}|)^{N_0-1}, \ \ \forall j \in \Z_{+}.
	\end{align*}
However the right-hand side goes to zero when $j\lra \infty$, giving us a contradiction.  Therefore $Q(\del_x)(\chi u)$ is not a distribution, as we intended to show. 
\end{proof}

\begin{Obs}
	One can also adapt Proposition~\ref{PropA1} to hyperfunctions as follows:
	\begin{Pro}\label{PropB1}
		For every $u\in \Cinf(U) \setminus \Com(U)$ there exists $Q$ an infinite order differential operator of class $\{p!\}$ such that $Qu \in \mathcal{B}(V) \setminus \D'(V)$ in some $V\subset U$ open.  
	\end{Pro}

	The key point is that when we take $\chi u \in \Cinf_c(U)$ and associate  the analytic functional
	\begin{align*}
		\lambda_{\chi u}(h)=\int_{\R^m} \chi(x) u(x) h(x) \dd x, \quad \forall h \in \mathcal{O}(C^m),
	\end{align*}
	we can follow the proof of Proposition~\ref{PropA1} to construct a hyperdifferential operator $Q$ such that analytic functional $Q(\del_x)  \lambda_{\chi u}$ does not represent a distribution.
\end{Obs}

\bibliographystyle{plain}
\bibliography{bibliography}

\end{document}